\documentclass[11pt,twoside]{article}
\newcommand{\ifims}[2]{#1} 
\newcommand{\ifAMS}[2]{#1}   
\newcommand{\ifau}[3]{#1}  

\newcommand{\ifbook}[2]{#1}   

\ifbook{

  }
  {

  }

\date{}
\def\thetitle{Sharp deviation bounds for quadratic forms}
\def\thanksa
{The author is partially supported by
Laboratory for Structural Methods of Data Analysis in Predictive Modeling, MIPT, 
RF government grant, ag. 11.G34.31.0073.
Financial support by the German Research Foundation (DFG) through the Collaborative 
Research Center 649 ``Economic Risk'' is gratefully acknowledged} 
\def\theruntitle{sharp deviation bounds for quadratic forms}

\def\theabstract{
This note presents sharp inequalities for deviation probability of a general 
quadratic form of a random vector \( \xiv \) with finite exponential moments.
The obtained deviation bounds are similar to the case of a Gaussian random vector.
The results are stated under general conditions and do not suppose any special 
structure of the vector \( \xiv \).
The obtained bounds are exact (non-asymptotic), all constants are explicit and
the leading terms in the bounds are sharp.
}

\def\kwdp{60F10}
\def\kwds{62F10}

\def\thekeywords{quadratic forms, deviation bounds}

\def\thankstitle{}
\def\authora{Vladimir Spokoiny}
\def\runauthora{spokoiny, v.}
\def\addressa{
    Weierstrass-Institute, \\ Humboldt University Berlin, \\ Moscow Institute of
    Physics and Technology
    \\
    Mohrenstr. 39, 10117 Berlin, Germany,    \\
    }
\def\emaila{spokoiny@wias-berlin.de}
\def\affiliationa{Weierstrass-Institute and Humboldt University Berlin}

\renewcommand{\(}{$\,}
\renewcommand{\)}{\,$}

\def\nquad{\hspace{-1cm}}
\def\eqdef{\stackrel{\operatorname{def}}{=}}

\newcommand{\cc}[1]{\mathscr{#1}}
\newcommand{\bb}[1]{\boldsymbol{#1}}

\renewcommand{\bar}[1]{\overline{#1}}
\renewcommand{\hat}[1]{\widehat{#1}}
\renewcommand{\tilde}[1]{\widetilde{#1}}

\renewcommand{\Gamma}{\varGamma}
\renewcommand{\Pi}{\varPi}
\renewcommand{\Sigma}{\varSigma}
\renewcommand{\Delta}{\varDelta}
\renewcommand{\Lambda}{\varLambda}
\renewcommand{\Psi}{\varPsi}
\renewcommand{\Phi}{\varPhi}
\renewcommand{\Theta}{\varTheta}
\renewcommand{\Omega}{\varOmega}
\renewcommand{\Xi}{\varXi}
\renewcommand{\Upsilon}{\varUpsilon}

\def\Var{\operatorname{Var}}

\def\tr{\operatorname{tr}}

\def\R{I\!\!R}
\def\E{I\!\!E}
\def\P{I\!\!P}

\def\kappa{\varkappa}

\def\T{\top}
\def\diag{\operatorname{diag}}

\def\loc{\operatorname{loc}}

\def\ev{\bb{e}}
\def\fv{\bb{f}}

\def\uv{\bb{u}}
\def\vv{\bb{v}}
\def\wv{\bb{w}}

\def\Yv{\bb{Y}}

\def\etav{\bb{\eta}}
\def\gammav{\bb{\gamma}}
\def\varepsilonv{\bb{\varepsilon}}

\def\upsilonv{\bb{\upsilon}}
\def\xiv{\bb{\xi}}
\def\zetav{\bb{\zeta}}

\def\CONST{\mathtt{C}}

\usepackage{amsmath,amssymb,amsthm}
\usepackage{natbib}
\usepackage{epsfig,graphicx}
\usepackage{comment}
\usepackage{color}
\usepackage{srcltx}
\usepackage[mathscr]{eucal}
\usepackage[math]{easyeqn}
\usepackage{etoolbox}

\ifims{
\textheight=23cm
\textwidth=14.8cm
\topmargin=0pt
\oddsidemargin=1.0cm
\evensidemargin=1.0cm
\linespread{1.3}
\renewenvironment{abstract}
    {\centerline{\textbf{Abstract}}\bigskip
      \begin{center}
       \begin{minipage}{11cm}
        \begin{small}
    }
    {   \end{small}
       \end{minipage}
      \end{center}
     \bigskip
    }

}{ 
}

\numberwithin{equation}{section}
\numberwithin{figure}{section}
\newcounter{example}[section]
\numberwithin{example}{section}
\newcounter{remark}[section]
\numberwithin{remark}{section}
\newtheorem{theorem}{Theorem}[section]

\newtheorem{lemma}[theorem]{Lemma}
\newtheorem{corollary}[theorem]{Corollary}

\newtheorem{exmp}[example]{Example}
\newtheorem{rmrk}[remark]{Remark}
\newenvironment{example}{\begin{exmp}\rm}{\end{exmp}}
\newenvironment{remark}{\begin{rmrk}\rm}{\end{rmrk}}

\begin{document}
\thispagestyle{empty}
\ifims{
\title{\thetitle}
\ifau{ 
  \author{
    \authora
    \ifdef{\thanksa}{\thanks{\thanksa}}{}
    \\[5.pt]
    \addressa \\
    \texttt{ \emaila}
  }
}
{  
  \author{
    \authora
    \ifdef{\thanksa}{\thanks{\thanksa}}{}
    \\[5.pt]
    \addressa \\
    \texttt{ \emaila}
    \and
    \authorb
    \ifdef{\thanksb}{\thanks{\thanksb}}{}
    \\[5.pt]
    \addressb \\
    \texttt{ \emailb}
  }
}
{   
  \author{
    \authora
    \ifdef{\thanksa}{\thanks{\thanksa}}{}
    \\[5.pt]
    \addressa \\
    \texttt{ \emaila}
    \and
    \authorb
    \ifdef{\thanksb}{\thanks{\thanksb}}{}
    \\[5.pt]
    \addressb \\
    \texttt{ \emailb}
    \and
    \authorc
    \ifdef{\thanksc}{\thanks{\thanksc}}{}
    \\[5.pt]
    \addressc \\
    \texttt{ \emailc}
  }
}

\maketitle
\pagestyle{myheadings}
\markboth
 {\hfill \textsc{ \small \theruntitle} \hfill}
 {\hfill
 \textsc{ \small
 \ifau{\runauthora}
      {\runauthora \, and \runauthorb}
      {\runauthora, \runauthorb, and \runauthorc}
 }
 \hfill}
\begin{abstract}
\theabstract
\end{abstract}

\ifAMS
    {\par\noindent\emph{AMS 2000 Subject Classification:} Primary \kwdp. Secondary \kwds}
    {\par\noindent\emph{JEL codes}: \kwdp}

} 
{ 
\begin{frontmatter}
\title{\thetitle\protect\thanksref{T1}}
\thankstext{T1}{\thankstitle}


\runtitle{\theruntitle}

\begin{aug}
\author{\authora\ead[label=e1]{\emaila}}
\address{\addressa \\
 \printead{e1}}
 \end{aug}

 \runauthor{\runauthora}
\affiliation{\affiliationa}

\begin{abstract}
\theabstract
\end{abstract}



\end{frontmatter}
} 

\def\ND{\cc{N}}
\def\Bernoulli{\mathrm{Bernoulli}}
\def\Vola{\mathrm{Vola}}
\def\Poisson{\mathrm{Poisson}}
\def\ag{\mathrm{ag}}
\def\glob{\operatorname{glob}}
\def\blk{\operatorname{block}}
\def\lin{\operatorname{lin}}
\def\cond{\, \big| \,}

\def\rdl{\epsilon}
\def\rd{\bb{\rdl}}
\def\rddelta{\delta}
\def\rdomega{\varrho}
\def\rddeltab{\rddelta^{*}}
\def\rhorb{\rhor^{*}}

\def\wv{\bb{w}}
\def\varthetav{\bb{\vartheta}}
\def\Lr{\breve{L}}
\def\zetavr{\breve{\zetav}}
\def\etavr{\breve{\etav}}
\def\xivr{\breve{\xiv}}

\def\rdb{\rd}
\def\rdm{\underline{\rdb}}

\def\taub{\tau_{\rdb}}
\def\taum{\tau_{\rdm}}
\def\kappab{\kappa_{\rd}}
\def\deltab{\delta_{\rd}}

\def\taubGP{\tau_{\rdb,\GP}}
\def\taumGP{\tau_{\rdm,\GP}}
\def\kappabGP{\kappa_{\rd,\GP}}
\def\deltabGP{\delta_{\rd,\GP}}

\def\rG{\rd,\GP}

\def\LinSp{\mathrm{L}}
\def\Id{I\!\!\!I}
\def\Ind{\operatorname{1}\hspace{-4.3pt}\operatorname{I}}

\def\BG{\mathcal{R}}
\def\bg{r}
\def\fmup{\phi}
\def\rg{r}
\def\uc{u_{c}}
\def\muc{\mu_{c}}
\def\mud{\mu_{0}}
\def\xxd{\xx_{0}}
\def\yyd{\yy_{0}}
\def\gmd{\gm_{0}}

\def\ms{m^{*}}
\def\Inv{A}
\def\InvT{\Inv^{\T}}
\def\Invt{\tilde{\Inv}}

\def\ssize{N}
\def\nsize{{n}}

\def\rhor{\omega}

\def\LT{L}
\def\LGP{\LT_{\GP}}
\def\La{\mathbb{L}}
\def\Lab{\La_{\rdb}}
\def\Lam{\La_{\rdm}}

\def\DP{D}
\def\DPc{\DP_{0}}
\def\DPb{\DP_{\rdb}}
\def\DPm{\DP_{\rdm}}

\def\LabGP{\La_{\rdb,\GP}}
\def\LamGP{\La_{\rdm,\GP}}

\def\DPbGP{\DP_{\rdb,\GP}}
\def\DPmGP{\DP_{\rdm,\GP}}
\def\riskbGP{\riskt_{\rdb,\GP}}

\def\gmi{\mathtt{b}}
\def\gmiid{\mathtt{g}_{1}}
\def\kullbi{\Bbbk}
\def\Thetasi{\Theta_{\loc}}
\def\rri{\mathtt{u}}
\def\rris{\rri_{0}}

\def\Ipc{\bb{\mathrm{f}}}
\def\IF{\Bbb{F}}
\def\IFc{\IF_{0}}
\def\IFb{\IF_{\rdb}}
\def\IFm{\IF_{\rdm}}

\def\DF{\cc{D}}
\def\DFc{\DF_{0}}
\def\DFb{\DF_{\rdb}}
\def\DFm{\breve{\DF}_{\rd}}
\def\DFm{\DF_{\rdm}}

\def\DPr{\breve{\DP}}
\def\VF{\cc{V}}
\def\VFc{\VF_{0}}

\def\HHc{\HH_{0}}
\def\HHb{\HH_{\rd}}
\def\HHm{\HH_{\rdm}}

\def\xib{\xi^{*}}
\def\xivb{\xiv_{\rdb}}
\def\xivm{\xiv_{\rdm}}
\def\CAm{\underline{\CA}}
\def\CAb{\CA}

\def\penr{\operatorname{pen}}
\def\pen{\mathfrak{t}}
\def\PEN{\operatorname{PEN}}
\def\RSS{\operatorname{RSS}}
\def\med{\operatorname{med}}

\def\ex{\mathrm{e}}
\def\entrl{\mathbb{Q}}
\def\entrlb{\entrl}
\def\entr{\entrl}

\def\kullb{\cc{K}} 
\def\kullbc{\kullb^{c}}

\def\gm{\mathtt{g}}
\def\gmc{\gm_{c}}
\def\gmb{\gm}
\def\gmbm{\gmb_{1}}

\def\yy{\mathtt{y}}
\def\yyc{\yy_{c}}
\def\xx{\mathtt{x}}
\def\xxc{\xx_{c}}
\def\tc{t_{c}}

\def\alp{\alpha}
\def\alpn{\rho}
\def\gmu{\mathfrak{a}}

\def\losst{\varrho}
\def\loss{\wp}
\def\lossp{u}
\def\closs{g}

\def\riskt{\cc{R}}
\def\emprisk{\ell}
\def\bias{b}
\def\bern{q}

\def\TT{\nsize}

\def\Pone{P}
\def\Pf{\P_{f(\cdot)}}
\def\Ef{\E_{f(\cdot)}}
\def\Ps{\P_{\thetas}}
\def\Es{\E_{\thetas}}
\def\Pu{\P_{\upsilons}}
\def\Eu{\E_{\upsilons}}

\def\Pvs{\P_{\thetavs}}
\def\Evs{\E_{\thetavs}}

\def\UPd{w}
\def\nunup{\nu_{1}}
\def\rru{\rr_{1}}
\def\rups{\rr_{0}}
\def\rupsb{\rups^{*}}
\def\rrf{\rr^{\flat}}

\def\smooths{\mathbb{S}}
\def\smooth{\smooths_{1}}

\def\elli{\bar{\ell}}

\def\K{K}

\def\Psir{\breve{\Psi}}

\def\af{a}
\def\afs{\af^{*}}

\def\kapla{\varkappa}

\newcommand{\mlew}[1]{\tilde{\thetav}_{#1}}
\newcommand{\mlea}[1]{\hat{\thetav}_{#1}}
\newcommand{\mluw}[1]{\tilde{\theta}_{#1}}
\newcommand{\mlua}[1]{\hat{\theta}_{#1}}
\newcommand{\penm}[1]{\boldsymbol{m}_{#1}}

\def\Pdom{\mu_{0}}
\def\PDOM{\bb{\mu}_{0}}
\def\EDOM{\E_{0}}

\def\mk{m}
\def\Mk{\cc{M}}
\def\SV{\cc{S}}

\def\Cs{E}
\def\Csd{\Cs^{\circ}}
\def\Ca{A}
\def\CS{\cc{E}}
\def\CA{\cc{A}}
\def\CAb{\CA_{\rd}}
\def\CAC{\CA_{\CoFu}}

\def\Ccb{m_{\rdb}}
\def\Ccm{m_{\rdm}}
\def\CcbGP{m_{\rdb,\GP}}
\def\CcmGP{m_{\rdm,\GP}}

\def\etas{\eta^{*}}

\def\omegav{\bb{\phi}}
\def\omegavs{\omegav^{*}}
\def\omegavc{\omegav'}

\def\nuvs{\nuv^{*}}
\def\nuvc{\nuv'}

\def\nunu{\nu_{0}}
\def\numu{\nu_{1}}
\def\nupi{\nu^{+}}
\def\nubu{\beta}

\def\nus{\nu}
\def\nusb{\nus}
\def\nusr{\nus^{\bracketing}}
\def\Nusb{\mathbb{N}}
\def\Nusr{\mathbb{N}^{\diamond}}

\def\dist{d}
\def\distd{\mathfrak{a}}

\def\hatk{\kappa}
\def\ko{k^{\circ}}

\def\qqq{\mathfrak{q}}
\def\ppp{{s}}
\def\Cqq{C(\qqq)}
\def\Cqqb{C^{\diamond}(\qqq)}
\def\Crho{C(\mrho)}
\def\Cqqm{\log(4)}
\def\Cqpr{(\qqq \rrp + \dimp / 2)}

\def\Cdima{\mathfrak{e}_{0}}
\def\Cdimb{\mathfrak{e}_{1}}
\def\cdima{\mathfrak{c}_{0}}
\def\cdimb{\mathfrak{c}_{1}}
\def\cdim{\mathfrak{c}}

\def\rdomega{\varrho}
\def\deltaD{\delta}
\def\alphai{\alpha_{1}}
\def\alphaii{\alpha_{2}}
\def\alphaiii{\alpha_{3}}
\def\alphaiv{\alpha_{4}}

\def\err{\diamondsuit}
\def\errbm{\bar{\err}_{\rdomega}}
\def\errm{\err_{\rdm}}
\def\errb{\err_{\rdb}}

\def\errbGP{\err_{\rdomega,\GP}}
\def\errmGP{\err_{\rdm,\GP}}
\def\errbmGP{\bar{\err}_{\rd,\GP}}

\def\errs{\err_{\rdomega}^{*}}
\def\deltas{\alpha}

\def\xivbGP{\xiv_{\rdb,\GP}}
\def\xivmGP{\xiv_{\rdm,\GP}}

\def\SP{S}
\def\GP{G}
\def\GPt{\GP_{0}}
\def\GPn{\GP_{1}}
\def\gp{g}
\def\gs{s}

\def\errbGP{\err_{\rdb,\GP}}
\def\errmGP{\err_{\rdm,\GP}}
\def\errpmGP{\err_{\GP}^{\pm}}

\def\LCS{\cc{C}}

\def\DPGP{\DP_{\GP}}
\def\thetavsGP{\thetavs_{\GP}}

\def\LL{\cc{L}}
\def\LLb{\LL^{*}}
\def\LLh{\cc{L}}

\def\YY{\cc{Y}}
\def\LP{L^{\circ}}

\def\modcnrd{\mathfrak{A}}

\def\pens{\pi}
\def\pnn{\mathfrak{g}}
\def\pnnd{\mathfrak{u}}
\def\pnndGP{\pnnd_{\GP}}

\def\confpr{\mathfrak{c}}
\def\confprb{\confpr^{*}}

\def\pn{\pens^{*}}
\def\penInt{\mathfrak{D}}
\def\penH{\mathbb{H}}
\def\pmu{\mathfrak{u}}
\def\Closs{\cc{R}}

\def\dimp{p}
\def\riskb{\riskt_{\rdb}}
\def\dimpp{\dimp+1}
\def\BB{I\!\!B}
\def\vA{\mathtt{v}}

\def\deficiency{\Delta}
\def\spread{\Delta}
\def\dimtotal{\dimp^{*}}

\def\thetav{\bb{\theta}}
\def\thetavs{\thetav^{*}}
\def\thetavc{\thetav'}
\def\thetavd{\thetav^{\circ}}
\def\thetavdc{\thetav^{\sharp}}
\def\dthetavs{\thetav,\thetavs}

\def\thetas{\theta^{*}}
\def\thetac{\theta'}
\def\thetad{\theta^{\circ}}
\def\thetab{\theta^{\dag}}
\def\thetavb{\thetav^{\dag}}

\def\vtheta{\vartheta}
\def\vthetav{\bb{\vtheta}}
\def\prior{\Pi}

\def\Gam{\Xi}
\def\Gam{\mathcal{S}}
\def\RG{R}
\def\Psu{\Upsilon}
\def\Phim{\breve{\Phi}}

\def\Proj{P}

\def\gammavs{\gammav^{*}}
\def\gammavd{\gammav^{\circ}}
\def\etavs{\etav^{*}}
\def\etavd{\etav^{\circ}}
\def\etavc{\etav'}

\def\taus{\tau_{0}}
\def\taup{\lceil \tau \rceil}

\def\sigmas{{\sigma^{*}}}
\def\Sigmas{\Sigma_{0}}

\def\upsilonc{\upsilon'}
\def\upsilond{\upsilon^{\circ}}
\def\upsilonp{{\upsilon}^{*}}
\def\upsilonm{{\upsilon}_{*}}
\def\upsilonvs{\upsilonv^{*}}
\def\upsilons{\upsilon^{*}}
\def\upsilonb{\bar{\upsilon}}
\def\upsilonvd{\upsilonv^{\circ}}

\def\ups{\bb{\upsilon}}
\def\upss{\ups_{0}}
\def\upsc{\ups^{\prime}}
\def\upsd{\ups^{\circ}}
\def\upsdc{\ups^{\sharp}}
\def\upsdu{\ups^{\flat}}

\def\Ups{\varUpsilon}
\def\Upsd{\Ups^{\circ}}
\def\Upss{\Ups_{\circ}}
\def\UpsP{\Ups^{c}}

\def\Thetas{\Theta_{0}}
\def\ThetasGP{\Theta_{0,\GP}}
\def\varthetav{\bb{\vartheta}}

\def\glink{g}

\def\fvs{\fv}
\def\fs{f}
\def\fb{\fv^{\dag}}

\def\uc{\uv'}
\def\ud{\uv^{\circ}}
\def\uvs{\uv^{*}}
\def\us{u^{*}}

\def\reps{\epsilon}
\def\eps{\epsilon}

\def\repsc{\reps_{0}}
\def\repsb{\reps^{*}}
\def\repsg{g}

\def\lu{\delta}
\def\lub{\bar{\lu}}

\def\Uu{U}
\def\UU{\cc{Y}}
\def\UUM{\cc{M}}
\def\UP{\cc{U}}
\def\up{\mathfrak{u}}

\def\VP{V}
\def\VPc{\VP_{0}}
\def\VPV{\cc{U}}
\def\VPVc{\cc{\VPV}_{0}}
\def\VPGP{\VP_{\GP}}
\def\VPSP{\VP_{\SP}}

\def\VV{H}
\def\GV{\cc{G}}
\def\GVS{S}

\def\VVb{\VV^{*}}
\def\VVc{\VV_{0}}
\def\vv{\bb{h}}
\def\vva{g}
\def\vp{\mathbf{v}}
\def\vpc{\vp_{0}}
\def\VVca{\VV}
\def\Vtt{H}

\def\DG{D}

\def\Vn{V_{0}}
\def\vn{v_{0}}

\def\norm{\mathfrak{c}}
\def\normc{\delta}
\def\norma{c}

\def\egridd{\cc{E}_{\delta}}
\def\penb{\varkappa}

\def\dotzeta{\dot{\zeta}}
\def\mes{\pi}
\def\mesl{\Lambda}
\def\cprr{F}

\def\lambdam{\gm_{1}}
\def\lambdaB{{\lambda}^{*}}
\def\lambdac{{\lambda'}}

\def\cla{{b}}
\def\fis{\mathfrak{a}}
\def\fiss{\fis_{1}}

\def\Vd{{V}}
\def\vd{\bar{v}}

\def\klim{k^{\circ}}
\def\midm{\mid \!}

\def\Ldrift{M}
\def\ldrift{m}
\def\mY{b}
\def\Lvar{D}
\def\lvar{\sigma}

\def\Mubcu{\Upsilon}
\def\Dthetav{\bb{u}}

\def\B{\cc{B}}
\def\BD{\B^{\circ}}
\def\BU{B}
\def\BI{\B^{*}}

\def\mub{\mu^{*}}
\def\mubc{\mu}
\def\mubcb{\mubc^{*}}
\def\Mubc{\mathbb{M}}
\def\Mubcb{\mathrm{M}}

\def\zzc{\zz_{c}}
\def\ww{w}
\def\wwc{\ww_{c}}

\def\norms{\circ} 
\def\rs{\rr_{\norms}}
\def\yys{\yy_{\norms}}
\def\xxs{\xx_{\norms}}
\def\zzs{\zz_{\norms}}
\def\uu{\mathtt{u}}
\def\uus{\uu_{\norms}}
\def\mus{\mu_{\norms}}
\def\gms{\gm_{\norms}}
\def\wws{\ww_{\circ}}

\def\srho{s}
\def\mrho{\varrho}

\def\Lmgf{\mathfrak{M}}
\def\Lmgfb{\Lmgf^{*}}

\def\lmgf{\mathfrak{m}}
\def\lmgfb{\lmgf^{*}}

\def\Expzeta{\mathfrak{N}}
\def\expzeta{\mathfrak{s}}

\def\rr{\mathtt{r}}
\def\rrb{\rr^{*}}
\def\rru{\rr_{\circ}}
\def\rrc{\rr'}
\def\rs{r_{*}}

\def\zz{\mathfrak{z}}
\def\zzb{\tilde{\zz}}
\def\tt{\mathfrak{t}}
\def\zb{z_{\rd}}
\def\zzg{\zz_{1}}
\def\zzQ{\zz_{0}}
\def\zzq{\zz}

\def\Cr{\mathfrak{c}}
\def\Crp{\mathfrak{C}}
\def\Crl{\mathfrak{r}}
\def\Crlp{\mathfrak{R}}
\def\Crlq{\cc{T}}
\def\Crlmu{\cc{M}}

\def\zetah{\zeta_{h}}
\def\GG{G}
\def\HH{H}
\def\pG{p}
\def\pH{q}
\def\hh{H^{*}}

\def\mubch{\mubc_{1}}
\def\rhoh{\rho_{1}}
\def\CoFuh{\CoFu_{1}}
\def\dimh{p_{1}}
\def\VPh{\VP_{1}}
\def\VPt{\VP_{0}}

\def\LLh{L_{1}}
\def\pnndh{\pnnd_{1}}

\def\LCS{C}
\def\Ac{A_{0}}
\def\Ab{A_{\rd}}
\def\DPrb{\DPr_{\rdb}}
\def\DPrm{\DPr_{\rdm}}
\def\Cb{\cc{C}_{\rdb}}
\def\Ub{\cc{U}_{\rdb}}
\def\zetavrb{\zetavr_{\rd}}
\def\xivrb{\breve{\xiv}_{\rd}}
\def\VPrb{\breve{\VP}_{\rdb}}
\def\Larb{\breve{\La}_{\rdb}}
\def\Larm{\breve{\La}_{\rdm}}

\def\deltav{\bb{\delta}}

\def\score{\nabla}
\def\scorer{\breve{\nabla}}

\def\LCS{C}
\def\Ac{A_{0}}
\def\Bc{B_{0}}
\def\AF{A}
\def\Ab{A_{\rdb}}
\def\Am{A_{\rdm}}
\def\DPrc{\DPr_{0}}
\def\DPrb{\DPr_{\rdb}}
\def\DPrm{\DPr_{\rdm}}
\def\Cb{\cc{C}_{\rdb}}
\def\Cm{\cc{C}_{\rdm}}
\def\Ub{\cc{U}_{\rdb}}
\def\deltav{\bb{\delta}}
\def\nuv{\bb{\nu}}
\def\xivrb{\breve{\xiv}_{\rd}}
\def\VPrb{\breve{\VP}_{\rdb}}
\def\Larb{\breve{\La}_{\rdb}}
\def\Lar{\breve{\La}}
\def\Larm{\breve{\La}_{\rdm}}
\def\VH{Q}
\def\VHc{\VH_{0}}
\def\zetavrm{\zetavr_{\rdm}}
\def\N{\mathbb{N}}

\def\Span{\operatorname{span}}
\def\Exc{{\square}}
\def\UUs{U_{\circ}}
\def\errbm{\errb^{*}}
\def\corrDF{\nu}
\def\BBr{\breve{\BB}}
\def\taua{\tau}
\def\AssId{\mathcal{I}}
\def\assId{\iota}
\def\AFD{\cc{A}}

\def\BanX{\cc{X}}
\def\basX{\ev}
\def\apprX{\alpha}
\def\fvs{\fv^{*}}
\def\lkh{\ell}
\def\Bc{B_{0}}
\def\dimn{\dimp_{\nsize}}
\def\betan{\beta_{\nsize}}


\def\xivGP{\xiv_{\GP}}
\def\dimA{\mathtt{p}}
\def\dimAGP{\dimA}
\def\dime{\dimA_{e}}
\def\dimG{\dimA_{\GP}}
\def\dimS{\dimA_{s}}
\def\nubm{\nu_{\rd}}
\def\uub{u_{\rd}}
\def\uubGP{u_{\rd,\GP}}

\def\priorden{\pi}
\def\xivGP{\xiv_{\GP}}
\def\dimAGP{\dimA}
\def\nubm{\nu_{\rd}}
\def\uub{u_{\rd}}
\def\uubGP{u_{\rd,\GP}}

\def\CR{\mathcal{C}}
\def\CRb{\CR_{\rdb}}
\def\vthetavb{\bar{\vthetav}}
\def\Covpost{\mathfrak{S}}

\def\Db{\DP_{+}}
\def\Dm{\DP_{-}}
\def\uvb{\uv_{+}}
\def\uvm{\uv_{-}}
\def\uud{\omega}
\def\taub{\delta}
\def\Lip{L}
\def\Xb{X_{+}}
\def\Xm{X_{-}}
\def\deltam{\delta_{-}}
\def\betauv{\delta}
\def\betab{\betauv_{1}}
\def\betaf{\betauv_{2}}
\def\upsv{\bb{\varkappa}}
\def\upsvb{\bar{\upsv}}
\def\rhob{\varrho}
\def\alpb{\alp_{1}}
\def\betap{\betauv_{3}}
\def\Ec{\E^{\circ}}
\def\ff{f}
\def\fpos{g}
\def\fneg{h}
\def\alpb{\alp_{+}}
\def\alpm{\alp_{-}}

\def\kappak{\kappa}
\def\kappas{\kappak^{*}}
\def\Kappak{\cc{K}}
\def\DPk{\DP_{\kappak}}
\def\VPk{\VP_{\kappak}}

\def\ts{s}
\def\tsv{\bb{\ts}}
\def\mm{\kappa}
\def\mmc{\mm'}
\def\mmd{\mm^{\circ}}
\def\mmo{\mm^{*}}
\def\mmmmo{\mm,\mmo}
\def\mmt{\tilde{\mm}}
\def\mma{\hat{\mm}}
\def\pp{z}

\def\LLL{L_{1}}
\def\LLr{L_{0}}
\def\muL{\mu_{1}}
\def\mur{\mu_{0}}

\def\LmgfL{\Lmgf_{1}}
\def\Lmgfr{\Lmgf_{0}}
\def\Lmgfm{\Lmgf_{1}}

\def\Kappa{\cc{K}}
\def\CoFu{\cc{C}}
\def\CoFuc{\CoFu_{0}}
\def\CoFub{\CoFu^{*}}
\def\CoFuL{\CoFu_{1}}
\def\CoFur{\CoFu_{0}}
\def\CAL{\CA_{1}}
\def\CAr{\CA_{0}}
\def\CAzz{\cc{A}}

\def\pnnL{\pnn_{1}}
\def\pnnr{\pnn_{0}}
\def\ttd{\delta}
\def\alphaL{\alpha_{1}}
\def\alphar{\alpha_{0}}
\def\alpharL{\alpha}
\def\rat{\mathfrak{t}}
\def\mquad{\nquad}
\def\zzL{\zz_{1}}
\def\zzr{\zz_{0}}

\def\mmset{\mathcal{I}}
\def\xex{u}
\def\dcm{q}
\def\dc{g}
\def\dcL{\dc_{1}}
\def\dcr{\dc_{0}}
\def\kk{k}

\def\cpen{\tau}

\def\dens{f}
\def\jj{j}
\def\JJ{\cc{J}}
\def\Zphi{Z}
\def\Zphiv{\bb{\Zphi}}

\def\nuu{\mathfrak{u}}
\def\nud{\mathfrak{u}_{0}}
\def\nun{c_{\nuu}}
\def\rhork{\kullb}
\def\GH{\mbox{GH}}
\def\HYP{\mbox{HYP}}
\def\NIG{\mbox{NIG}}
\def\IR{{\rm I\!R}}
\def\taggr{b}
\def\penm{\boldsymbol{m}}
\def\Crlp{\cc{R}}


\section{Introduction}
\label{Chgqform}
\label{Sprobabquad}

This paper presents a number of deviation probability bounds for a quadratic form
\( \| \xiv \|^{2} \) or more generally \( \| \BB \xiv \|^{2} \) of a random 
\( \dimp \) vector \( \xiv \) satisfying a general exponential moment condition.
Such quadratic forms arise in many problems. 
We mainly focus on statistical applications such that hypothesis testing for linear 
models or linear model selection. 
We refer to \cite{massart2003} for an extensive overview and numerous results on 
probability bounds and their applications in statistical model selection. 
Limit theorems for quadratic forms can be found e.g. in \cite{GoTi1999} and 
\cite{HoSh1999}. 
Some concentration bounds for U-statistics are available in 
\cite{Bret1999},
\cite{Gine2000},
\cite{HouRe2003}.
We also refer to \cite{Ba2010} for a number of statistical problems relying on such 
deviation bounds.

If \( \xiv \) is standard normal then \( \| \xiv \|^{2} \) is chi-squared with 
\( \dimp \) degrees of freedom. 
We aim to extend this behavior to the case of a general vector \( \xiv \)
satisfying the following exponential moment condition:
\begin{EQA}[c]
    \log \E \exp\bigl( \gammav^{\T} \xiv \bigr) 
    \le 
    \| \gammav \|^{2}/2,
    \qquad 
    \gammav \in \R^{\dimp}, \, \| \gammav \| \le \gm .
\label{expgamgm}
\end{EQA}
Here \( \gm \) is a positive constant which appears 
to be very important in our results.
Namely, it determines the frontier between the Gaussian and non-Gaussian type 
deviation bounds. 
Our first result shows that under \eqref{expgamgm} the deviation bounds for the 
quadratic form \( \| \xiv \|^{2} \) are essentially the same as in the Gaussian 
case, if the value \( \gm^{2} \) exceed \( \CONST \dimp \) for a fixed constant 
\( \CONST \).
Further we extend the result to the case of a more general form 
\( \| \BB \xiv \|^{2} \).
An important advantage of the approach of this paper which differs it from all the 
previous studies is that there is no any 
additional conditions on the structure or origin of the vector \( \xiv \).
For instance, we do not assume that \( \xiv \) is a sum of independent or weakly 
dependent random variables, or components of \( \xiv \) are independent.
The results are exact stated in a non-asymptotic fashion, all the constants are 
explicit and the leading terms are sharp.

As a motivating example, we consider a linear regression model 
\( \Yv = \Psi^{\T} \thetav + \varepsilonv \) 
in which the error vector \( \varepsilon \) is zero mean.
The ordinary least square estimator \( \tilde{\thetav} \) for the parameter vector 
\( \thetav \) reads as 
\begin{EQA}[c]
    \tilde{\thetav}
    =
    \bigl( \Psi \Psi^{\T} \bigr)^{-1} \Psi \Yv
\label{ttPsPsYv}
\end{EQA}    
and it can be viewed as the maximum likelihood estimator in a Gaussian linear model 
with a diagonal covariance matrix, that is, 
\( \Yv \sim \ND(\Psi^{\T} \thetav, \sigma^{2} \Id_{\nsize}) \). 
Define the \( \dimp \times \dimp \) matrix
\begin{EQA}[c]
    \DPc^{2}
    \eqdef
    \Psi \Psi^{\T},
\label{DPcVPsqf}
\end{EQA}    
Then
\begin{EQA}[c]
    \DPc (\tilde{\thetav} - \thetavs)
    =
    \DPc^{-1} \zetav
\label{DPcttsqf}
\end{EQA}    
with \( \zetav \eqdef \Psi \Yv \). 
The likelihood ratio test statistic for this problem is exactly 
\( \| \DPc^{-1} \zetav \|^{2}/2 \).
Similarly, the model selection procedure is based on comparing such quadratic forms 
for different matrices \( \DPc \); see e.g. \cite{Ba2010}.

Now we indicate how this situation can be reduced to a bound for a vector \( \xiv \) 
satisfying the condition \eqref{expgamgm}.
Suppose for simplicity that the errors \( \varepsilon_{i} \) are independent and 
have exponential moments. 

\begin{description}
\item[\( \bb{(e_{1})} \)] 
    \emph{ There exist some constants \( \nunu \) and \( \gmiid > 0 \), 
    and for every \( i \)
    a constant \( \expzeta_{i} \) such that 
    \( \E \bigl( \varepsilon_{i}/\expzeta_{i} \bigr)^{2} \le 1 \) and
    } 
\begin{EQA}[c]
    \log \E \exp\bigl(  {\lambda \varepsilon_{i}}/{\expzeta_{i}} \bigr)
    \le 
    \nunu^{2} \lambda^{2} / 2,
    \qquad 
    |\lambda| \le \gmiid .
\label{expzetanunu}
\end{EQA}
\end{description}

Here \( \gmiid \) is a fixed positive constant.
One can show that if this condition is fulfilled for some \( \gmiid > 0 \) and 
a constant \( \nunu \ge 1 \), then one can get a similar condition with 
\( \nunu \) arbitrary close to one and \( \gmiid \) slightly decreased. 
A natural candidate for \( \expzeta_{i} \) is \( \sigma_{i} \) where 
\( \sigma_{i}^{2} = \E \varepsilon_{i}^{2} \) is the variance of 
\( \varepsilon_{i} \).
Under \eqref{expzetanunu}, introduce a \( \dimp \times \dimp \) matrix \( \VPc \) 
defined by 
\begin{EQA}[c]
    \VPc^{2} 
    \eqdef 
    \sum \expzeta_{i}^{2} \Psi_{i} \Psi_{i}^{\T} .
\label{VPlinregr}
\end{EQA}    
Define also
\begin{EQA}
    \xiv
    &=&
    \VPc^{-1} \Psi \Yv,
    \\
    N^{-1/2}  
    & \eqdef &
    \max_{i} \sup_{\gammav \in \R^{p}} 
    \frac{\expzeta_{i} |\Psi_{i}^{\T} \gammav|}{\| \VPc \gammav \|} \, .
\label{CPsiexp}
\end{EQA}
Simple calculation shows that for \( \| \gammav \| \le \gmb = \gmiid N^{1/2} \)
\begin{EQA}[c]
    \log \E \exp\bigl( \gammav^{\T} \xiv \bigr) 
    \le 
    \nunu^{2} \| \gammav \|^{2}/2,
    \qquad 
    \gammav \in \R^{\dimp}, \, \| \gammav \| \le \gm .
\label{expgamgm0ex}
\end{EQA}
We conclude that \eqref{expgamgm} is nearly fulfilled under \( (e_{1}) \) and 
moreover, the value \( \gm^{2} \) is proportional to the effective sample size 
\( N \).
The results of the paper allow to get a nearly \( \chi^{2} \)-behavior of the test 
statistic \( \| \xiv \|^{2} \) which is a finite sample version of the famous Wilks 
phenomenon; see e.g. \cite{FaZh2001,FaHu2005}, \cite{BoMa2011}. 

\medskip
The paper is organized as follows.
Section~\ref{SGaussqf} reminds the classical results about deviation probability of 
a Gaussian quadratic form. 
These results are presented only for comparison and to make the paper selfcontained. 

Section~\ref{SchiquadE}
studies the probability of the form \( \P\bigl( \| \xiv \| > \yy \bigr) \) 
under the condition 
\begin{EQA}[c]
    \log \E \exp\bigl( \gammav^{\T} \xiv \bigr) 
    \le 
    \nunu^{2} \| \gammav \|^{2}/2,
    \qquad 
    \gammav \in \R^{\dimp}, \,\, \| \gammav \| \le \gm .
\label{expgamgm0}
\end{EQA}
The general case can be reduced to \( \nunu = 1 \) by rescaling 
\( \xiv \) and \( \gm \):
\begin{EQA}[c]
    \log \E \exp\bigl( \gammav^{\T} \xiv / \nunu \bigr) 
    \le 
    \| \gammav \|^{2}/2,
    \qquad 
    \gammav \in \R^{\dimp}, \,\, \| \gammav \| \le \nunu \gm 
\label{expgamgmm}
\end{EQA}
that is, \( \nunu^{-1} \xiv \) fulfills \eqref{expgamgm} with a slightly increased
\( \gm \).

The result is extended to the case of a general quadratic form 
in Section~\ref{Sbqf}.
Some more extension motivated by different statistical problems are given in 
Section~\ref{Schi2norm} and Section~\ref{SBernstqf}.
All the proofs are collected in the Appendix.

\section{Gaussian case}
\label{SGaussqf}
Our benchmark will be a deviation bound for \( \| \xiv \|^{2} \) for a standard 
Gaussian vector \( \xiv \).
The ultimate goal is to show that under \eqref{expgamgm} 
the norm of the vector \( \xiv \) exhibits behavior  
expected for a Gaussian vector, at least in the region of moderate deviations. 
For the reason of comparison, we begin by stating the result for a Gaussian vector 
\( \xiv \).

\begin{theorem}
\label{TxivG}
Let \( \xiv \) be a standard normal vector in \( \R^{\dimp} \).
Then for any \( u > 0 \), it holds
\begin{EQA}
\label{logPmunuu}
    \P\bigl( \| \xiv \|^{2} > \dimp + u \bigr)
    & \le &
    \exp\bigl\{ - (\dimp/2) \fmup(u/\dimp) \bigr] \bigr\} 
\end{EQA}
with 
\begin{EQA}[c]
    \fmup(t)
    \eqdef
    t - \log(1+t) .
\label{fmupd}
\end{EQA}   
Let \( \fmup^{-1}(\cdot) \) stand for the inverse of \( \fmup(\cdot) \). 
For any \( \xx \), 
\begin{EQA}[c]
    \P\bigl( 
        \| \xiv \|^{2}
        > \dimp + \fmup^{-1}(2\xx/\dimp)
    \bigr)
    \le 
    \exp(- \xx) .
\label{Ptttlg0}
\end{EQA}
This particularly yields with \( \kappa = 6.6 \)
\begin{EQA}[c]
    \P\bigl( 
        \| \xiv \|^{2}
        > \dimp + \sqrt{\kappa \xx \dimp} \vee (\kappa \xx) 
    \bigr)
    \le 
    \exp(- \xx) .
\label{Ptttlg}
\end{EQA}    
\end{theorem}

This is a simple version of a well known result and we present it only for 
comparison with the non-Gaussian case.
The message of this result is that the squared norm of the Gaussian vector 
\( \xiv \) concentrates around the value \( \dimp \) and the deviation over the 
level \( \dimp + \sqrt{\xx \dimp} \) are exponentially small in \( \xx \).

A similar bound can be obtained for a norm of the vector \( \BB \xiv \) where 
\( \BB \) is some given matrix.
For notational simplicity we assume that \( \BB \) is symmetric. 
Otherwise one should replace it with \( (\BB^{\T} \BB)^{1/2} \).

\begin{theorem}
\label{TexpbLGA}
Let \( \xiv \) be standard normal in \( \R^{\dimp} \).
Then for every \( \xx > 0 \) and any symmetric matrix \( \BB \), 
it holds with \( \dimA = \tr(\BB^{2}) \), 
\( \vA^{2} = 2 \tr(\BB^{4}) \), and \( a^{*} = \| \BB^{2} \|_{\infty} \)
\begin{EQA}[c]
    \P\bigl( 
        \| \BB \xiv \|^{2}
        > \dimA + (2 \vA \xx^{1/2}) \vee (6 a^{*} \xx) 
    \bigr)
    \le 
    \exp(- \xx) .
\label{PtttLGA}
\end{EQA}    
\end{theorem}

Below we establish  similar bounds for a non-Gaussian vector \( \xiv \) 
obeying \eqref{expgamgm}.

\section{A bound for the \( \ell_{2} \)-norm}
\label{SchiquadE}
This section presents a general exponential bound for the probability 
\( \P\bigl( \| \xiv \| > \yy \bigr) \) under \eqref{expgamgm}.
The main result tells us that if \( \yy \) is not too large, 
namely if \( \yy \le \yyc \) with \( \yyc^{2} \asymp \gm^{2} \), then 
the deviation probability is essentially the same as in the Gaussian case.

To describe the value \( \yyc \), introduce the following notation.
Given \( \gm \) and \( \dimp \),
define the values \( w_{0} = \gm \dimp^{-1/2} \) and 
\( \wwc \) by the equation
\begin{EQA}[c]
    \frac{\wwc(1+\wwc)}{(1+\wwc^{2})^{1/2}}
    =
    w_{0} 
    =
    \gm \dimp^{-1/2}.
\label{wc212}
\end{EQA}    
It is easy to see that \( w_{0}/\sqrt{2} \le \wwc \le w_{0} \). 
Further define 
\begin{EQA}
    \muc
    & \eqdef &
    \wwc^{2}/(1+\wwc^{2})
    \\
    \yyc 
    & \eqdef &
    \sqrt{(1 + \wwc^{2}) \dimp} ,
\label{yyc1wc}
    \\
    \xxc 
    & \eqdef &
    0.5 \dimp \bigl[ \wwc^{2} - \log\bigl( 1 + \wwc^{2} \bigr) \bigr].
\label{zcgmp}
\end{EQA}
Note that for \( \gm^{2} \ge \dimp \), the quantities \( \yyc \) and \( \xxc \) can be 
evaluated as 
\( \yyc^{2} \ge \wwc^{2} \dimp \ge \gm^{2}/2 \) and 
\( \xxc \gtrsim \dimp \wwc^{2}/2 \ge \gm^{2}/4 \).

\begin{theorem}
\label{TxivqLD} 
Let \( \xiv \in \R^{\dimp} \) fulfill \eqref{expgamgm}. 
Then it holds for each \( \xx \le \xxc \) 
\begin{EQA}
    \P\bigl( 
        \| \xiv \|^{2} > \dimp + \sqrt{\kappa \xx \dimp} \vee (\kappa \xx) , \,\,
        \| \xiv \| \le \yyc
    \bigr)
    & \le &
    2 \exp( - \xx ),
\label{expxibo}
\end{EQA}    
where \( \kappa = 6.6 \).
Moreover, for \( \yy \ge \yyc \), it holds with 
\( \gmc = \gm - \sqrt{\muc \dimp} = \gm \wwc/(1+\wwc) \)
\begin{EQA}
    \P\bigl( \| \xiv \| > \yy \bigr)
    & \le &
    8.4 \exp\bigl\{ - \gmc \yy/2 - (\dimp/2) \log(1 - \gmc/\yy) \bigr\}
    \\
    & \le &
    8.4 \exp\bigl\{ - \xxc - \gmc (\yy - \yyc)/2 \bigr\}.
\label{Pexp2xit}
\end{EQA}
\end{theorem}

The statements of Theorem~\ref{TxivqLDA} can be simplified 
under the assumption \( \gm^{2} \ge \dimp \).
\begin{corollary}
\label{CTxivqLDA}
Let \( \xiv \) fulfill \eqref{expgamgm} and \( \gm^{2} \ge \dimp \). 
Then it holds for \( \xx \le \xxc \) 
\begin{EQA}
\label{Pzzxxp}
    \P\bigl( \| \xiv \|^{2} \ge \zz(\xx,\dimp) \bigr)
    & \le &
    2 \ex^{-\xx} + 8.4 \ex^{-\xxc},
    \\
    \zz(\xx,\dimp)
    & \eqdef &
    \begin{cases}
        \dimp + \sqrt{\kappa \xx \dimp}, &  \xx \le \dimp/\kappa , \\
        \dimp + \kappa \xx & \dimp/\kappa < \xx \le \xxc ,
    \end{cases}
\label{zzxxppd}
\end{EQA}    
with \( \kappa = 6.6 \).
For \( \xx > \xxc \)
\begin{EQA}
    \P\bigl( \| \xiv \|^{2} \ge \zzc(\xx,\dimp) \bigr)
    & \le &
    8.4 \ex^{-\xx},
    \qquad 
    \zzc(\xx,\dimp)
    \eqdef 
    \bigl| \yyc + 2 (\xx - \xxc)/\gmc \bigr|^{2} .
\label{zzcxxppd}
\end{EQA}    
\end{corollary}

This result implicitly assumes that \( \dimp \le \kappa \xxc \) which is fulfilled 
if \( w_{0}^{2} = \gm^{2}/\dimp \ge 1 \):
\begin{EQA}[c]
    \kappa \xxc
    =
    0.5 \kappa \bigl[ w_{0}^{2} - \log(1 + w_{0}^{2}) \bigr] \dimp
    \ge 
    3.3 \bigl[ 1 - \log(2) \bigr] \dimp
    > 
    \dimp .
\label{dm2dimp} 
\end{EQA}    
For \( \xx \le \xxc \), the function \( \zz(\xx,\dimp) \) mimics the quantile 
behavior of the chi-squared distribution \( \chi^{2}_{\dimp} \) with \( \dimp \) 
degrees of freedom.
Moreover, increase of the value \( \gm \) yields a growth of the sub-Gaussian zone.
In particular, for \( \gm = \infty \), a general quadratic form \( \| \xiv \|^{2} \) 
has under \eqref{expgamgm} the same tail behavior as in the Gaussian case.

Finally, in the large deviation zone \( \xx > \xxc \) the deviation probability 
decays as \( \ex^{-c \xx^{1/2}} \) for some fixed \( c \).
However, if the constant \( \gm \) in the condition \eqref{expgamgm} is sufficiently large 
relative to \( \dimp \), then \( \xxc \) is large as well and the large deviation 
zone \( \xx > \xxc \) can be ignored at a small price of \( 8.4 \ex^{-\xxc} \) and 
one can focus on the deviation bound described by \eqref{Pzzxxp} and \eqref{zzxxppd}.

\section{A bound for a quadratic form}
\label{Sbqf}
Now we extend the result to more general bound for 
\( \| \BB \xiv \|^{2} = \xiv^{\T} \BB^{2} \xiv \) with a given matrix 
\( \BB \) and a vector \( \xiv \) obeying the condition \eqref{expgamgm}. 
Similarly to the Gaussian case we assume that \( \BB \) is symmetric. 
Define important characteristics of \( \BB \)
\begin{EQA}[c]
    \dimA = \tr (\BB^{2}) , 
    \qquad 
    \vA^{2} = 2 \tr(\BB^{4}),
    \qquad 
    \lambdaB \eqdef \| \BB^{2} \|_{\infty} \eqdef \lambda_{\max}(\BB^{2}) .
\label{dimAvAlb}
\end{EQA}
For simplicity of formulation we suppose that \( \lambdaB = 1 \),
otherwise one has to replace \( \dimA \) and \( \vA^{2} \) with 
\( \dimA/\lambdaB \) and \( \vA^{2}/\lambdaB \).

Let \( \gm \) be shown in \eqref{expgamgm}.
Define similarly to the \( \ell_{2} \)-case 
\( \wwc \) by the equation
\begin{EQA}[c]
    \frac{\wwc(1+\wwc)}{(1+\wwc^{2})^{1/2}}
    =
    \gm \dimA^{-1/2} .
\label{wc212A}
\end{EQA}
Define also \( \muc = \wwc^{2}/(1+\wwc^{2}) \wedge 2/3 \).
Note that \( \wwc^{2} \ge 2 \) implies \( \muc = 2/3 \).
Further define
\begin{EQA}
    \yyc^{2} = (1 + \wwc^{2}) \dimA,
    \qquad 
    2 \xxc
    & = &
    \muc \yyc^{2} + \log \det\{ \Id_{\dimp} - \muc \BB^{2} \} .
\label{xxcyycA}
\end{EQA}
Similarly to the case with \( \BB = \Id_{\dimp} \), under the condition 
\( \gm^{2} \ge \dimA \), one can bound 
\( \yyc^{2} \ge \gm^{2}/2 \) and \( \xxc \gtrsim \gm^{2}/4 \).

\begin{theorem}
\label{TxivqLDA}
Let a random vector \( \xiv \) in \( \R^{\dimp} \) fulfill \eqref{expgamgm}.
Then for each \( \xx < \xxc \)
\begin{EQA}
    \P\bigl( 
        \| \BB \xiv \|^{2} > \dimA + (2 \vA \xx^{1/2}) \vee (6 \xx), \,\, 
        \| \BB \xiv \| \le \yyc
    \bigr)
    & \le &
    2 \exp( - \xx ) .
\label{expxiboA}
\end{EQA}    
Moreover, for \( \yy \ge \yyc \), with
\( \gmc = \gm - \sqrt{\muc \dimA} = \gm \wwc/(1+\wwc) \),
it holds
\begin{EQA}
    \P\bigl( \| \BB \xiv \| > \yy \bigr)
    & \le &
    8.4 \exp\bigl( - \xxc - \gmc (\yy - \yyc)/2 \bigr) .
\label{expxibogA}
\end{EQA}
\end{theorem}

Now we describe the value \( \zz(\xx,\BB) \) ensuring a small value for the large deviation probability 
\( \P\bigl( \| \BB \xiv \|^{2} > \zz(\xx,\BB) \bigr) \).
For ease of formulation, we suppose that \( \gm^{2} \ge 2 \dimA \) yielding 
\( \muc^{-1} \le 3/2 \).
The other case can be easily adjusted. 

\begin{corollary}
\label{CTxivqLDAB}
Let \( \xiv \) fulfill \eqref{expgamgm} with \( \gm^{2} \ge 2 \dimA \). 
Then it holds for \( \xx \le \xxc \) with \( \xxc \) from \eqref{xxcyycA}:
\begin{EQA}
\label{PzzxxpB}
    \P\bigl( \| \BB \xiv \|^{2} \ge \zz(\xx,\BB) \bigr)
    & \le &
    2 \ex^{-\xx} + 8.4 \ex^{-\xxc},
    \\
    \zz(\xx,\BB)
    & \eqdef &
    \begin{cases}
        \dimA + 2 \vA \xx^{1/2}, &  \xx \le \vA/18 , \\
        \dimA + 6 \xx & \vA/18 < \xx \le \xxc .
    \end{cases}
\label{zzxxppdB}
\end{EQA}    
For \( \xx > \xxc \)
\begin{EQA}
    \P\bigl( \| \BB \xiv \|^{2} \ge \zzc(\xx,\BB) \bigr)
    & \le &
    8.4 \ex^{-\xx},
    \qquad 
    \zzc(\xx,\BB)
    \eqdef 
    \bigl| \yyc + 2 (\xx - \xxc)/\gmc \bigr|^{2} .
\label{zzcxxppdB}
\end{EQA}    
\end{corollary}

\section{Rescaling and regularity condition}
\label{SLDQFr}
The result of Theorem~\ref{TxivqLDA} can be extended to a more general situation 
when the condition \eqref{expgamgm} is fulfilled for a vector \( \zetav \) rescaled 
by a matrix \( \VPc \).
More precisely, let the random \( \dimp \)-vector \( \zetav \) fulfills for some 
\( \dimp \times \dimp \) matrix \( \VPc \) the condition 
\begin{EQA}[c]
\label{expzetaclocz} 
\sup_{\gammav \in \R^{\dimp}} 
    \log \E \exp\Bigl( 
        \lambda \frac{\gammav^{\T} \zetav}{\| \VPc \gammav \|} 
    \Bigr) 
    \le 
    \nunu^{2} \lambda^{2} / 2,
    \qquad 
    |\lambda| \le \gm,
\end{EQA}
with some constants \( \gm > 0 \), \( \nunu \ge 1 \).
Again, a simple change of variables reduces the case of an arbitrary \( \nunu \ge 1 \)
to \( \nunu = 1 \).
Our aim is to bound the squared norm \( \| \DPc^{-1} \zetav \|^{2} \) of a vector 
\( \DPc^{-1} \zetav \) for another \( \dimp \times \dimp \) positive symmetric matrix 
\( \DPc^{2} \).
Note that condition \eqref{expzetaclocz} implies \eqref{expgamgm} for the rescaled 
vector \( \xiv = \VPc^{-1} \zetav \).
This leads to bounding the quadratic form 
\( \| \DPc^{-1} \VPc \xiv \|^{2} = \| \BB \xiv \|^{2} \) with 
\( \BB^{2} = \DPc^{-1} \VPc^{2} \DPc^{-1} \). 
It obviously holds
\begin{EQA}[c]
    \dimA 
    =
    \tr(\BB^{2})
    =
    \tr (\DPc^{-2} \VPc^{2}) .
\label{dimAVPDP}
\end{EQA}    
Now we can apply the result of Corollary~\ref{CTxivqLDAB}.

\begin{corollary}
\label{CTxivqLDDV}
Let \( \zetav \) fulfill \eqref{expzetaclocz} with 
some \( \VPc \) and \( \gm \).
Given \( \DPc \), define \( \BB^{2} = \DPc^{-1} \VPc^{2} \DPc^{-1} \), and let
\( \gm^{2} \ge 2 \dimA \).
Then it holds for \( \xx \le \xxc \) with \( \xxc \) from \eqref{xxcyycA}:
\begin{EQA}
\label{PzzxxpB}
    \P\bigl( \| \DPc^{-1} \zetav \|^{2} \ge \zz(\xx,\BB) \bigr)
    & \le &
    2 \ex^{-\xx} + 8.4 \ex^{-\xxc},
\label{zzxxppdBDV}
\end{EQA}    
with \( \zz(\xx,\BB) \) from \eqref{zzxxppdB}. 
For \( \xx > \xxc \)
\begin{EQA}
    \P\bigl( \| \DPc^{-1} \zetav \|^{2} \ge \zzc(\xx,\BB) \bigr)
    & \le &
    8.4 \ex^{-\xx},
    \qquad 
    \zzc(\xx,\BB)
    \eqdef 
    \bigl| \yyc + 2 (\xx - \xxc)/\gmc \bigr|^{2} .
\label{zzcxxppdBDV}
\end{EQA}    
\end{corollary}

In the \emph{regular} case with \( \DPc \ge \fis \VPc \) for some 
\( \fis > 0 \), one obtains 
\( \| \BB \|_{\infty} \le \fis^{-1} \) and
\begin{EQA}[c]
    \vA^{2}
    =
    2 \tr(\BB^{4})
    \le 
    2 \fis^{-2} \dimA .
\label{vAfisdimA}
\end{EQA}

\section{A chi-squared bound with norm-constraints}
\label{Schi2norm}
This section extends the results to the case when the bound \eqref{expgamgm} requires 
some other conditions than the \( \ell_{2} \)-norm of the vector \( \gammav \).
Namely, we suppose that 
\begin{EQA}[c]
    \log \E \exp\bigl( \gammav^{\T} \xiv \bigr) 
    \le 
    \| \gammav \|^{2}/2,
    \qquad 
    \gammav \in \R^{\dimp}, \,\, \| \gammav \|_{\norms} \le \gms ,
\label{expgamgmmn}
\end{EQA}
where \( \| \cdot \|_{\norms} \) is a norm which differs from the usual 
Euclidean norm.
Our driving example is given by the sup-norm case with 
\( \| \gammav \|_{\norms} \equiv \| \gammav \|_{\infty} \).
We are interested to check whether the previous results of Section~\ref{SchiquadE} 
still apply.
The answer depends on how massive the set 
\( \CA(\rg) = \{ \gammav: \| \gammav \|_{\norms} \le \rg \} \) is in terms of the standard Gaussian 
measure on \( \R^{\dimp} \). 
Recall that the quadratic norm \( \| \varepsilonv \|^{2} \) of a standard Gaussian 
vector \( \varepsilonv \) in \( \R^{\dimp} \) concentrates around \( \dimp \) 
at least for \( \dimp \) large. 
We need a similar concentration property for the norm \( \| \cdot \|_{\norms} \).
More precisely, we assume for a fixed \( \rs \) that 
\begin{EQA}[c]
    \P\bigl( \| \varepsilonv \|_{\norms} \le \rs \bigr)
    \ge 
    1/2,
    \qquad 
    \varepsilonv \sim \ND(0,\Id_{\dimp}) .
\label{epsuvgmr}
\end{EQA}    
This implies for any value \( \uus > 0 \) and all \( \uv \in \R^{\dimp} \) with 
\( \| \uv \|_{\norms} \le \uus \) that 
\begin{EQA}[c]
    \P\bigl( \| \varepsilonv - \uv \|_{\norms} \le \rs + \uus \bigr)
    \ge 
    1/2,
    \qquad 
    \varepsilonv \sim \ND(0,\Id_{\dimp}) .
\label{epsuvgmrr}
\end{EQA}    
For each \( \zz > \dimp \), consider 
\begin{EQA}[c]
    \mu(\zz) = (\zz - \dimp)/\zz .
\label{muzzdp}
\end{EQA}    
Given \( \uus \), denote by \( \zzs = \zzs(\uus) \) the root of the equation
\begin{EQA}[c]
    \frac{\gms}{\mu(\zzs)}  - \frac{\rs}{\mu^{1/2}(\zzs)} 
    = 
    \uus .
\label{gmmuyynn}
\end{EQA}    
One can easily see that this value exists and unique 
if \( \uus \ge \gms - \rs \) and it can be defined as the largest \( \zz \)
for which \( \frac{\gms}{\mu(\zz)}  - \frac{\rs}{\mu^{1/2}(\zz)} \ge \uus \).
Let \( \mus = \mu(\zzs) \) be the corresponding \( \mu \)-value.
Define also \( \xxs \) by 
\begin{EQA}[c]
    2 \xxs = \mus \zzs + \dimp \log(1 - \mus) .
\label{xxdgg}
\end{EQA}    
If \( \uus < \gms - \rs \), then set \( \zzs = \infty \), \( \xxs = \infty \).

\begin{theorem}
\label{TxivqLDrg}
Let a random vector \( \xiv \) in \( \R^{\dimp} \) fulfill \eqref{expgamgmmn}.
Suppose \eqref{epsuvgmr} and let, given \( \uus \), the value \( \zzs \) be defined 
by \eqref{gmmuyynn}.
Then it holds for any \( u > 0 \) 
\begin{EQA}
\label{logPmunuunc}
    \P\bigl( \| \xiv \|^{2} > \dimp + u, 
        \| \xiv \|_{\norms} \le \uus 
    \bigr)
    & \le &
    2 \exp\bigl\{ - (\dimp/2) \fmup(u) \bigr] \bigr\} .
\end{EQA}
yielding for \( \xx \le \xxs \) 
\begin{EQA}
    \P\bigl( 
        \| \xiv \|^{2} > \dimp + \sqrt{\kappa \xx \dimp} \vee (\kappa \xx), \,
        \| \xiv \|_{\norms} \le \uus
    \bigr)
    & \le &
    2 \exp( - \xx ),
\label{expxibon}
\end{EQA}    
where \( \kappa = 6.6 \).
Moreover, for \( \zz \ge \zzs \), it holds 
\begin{EQA}
    \P\bigl( \| \xiv \|^{2} > \zz, \,
        \| \xiv \|_{\norms} \le \uus \bigr)
    & \le &
    2 \exp\bigl\{ - \mus \zz/2 - (\dimp/2) \log(1 - \mus) \bigr\}
    \\
    & = &
    2 \exp\bigl\{ - \xxs - \gms (\zz - \zzs)/2 \bigr\}.
\label{Pexp2xitrg}
\end{EQA}
\end{theorem}

It is easy to check that the result continues to hold for the norm of
\( \Pi \xiv \) for a given sub-projector \( \Pi \) in \( \R^{\dimp} \) 
satisfying \( \Pi = \Pi^{\T} \), \( \Pi^{2} \le \Pi \).
As above, denote \( \dimA \eqdef \tr (\Pi^{2}) \), 
\( \vA^{2} \eqdef 2 \tr (\Pi^{4}) \).
Let \( \rs \) be fixed to ensure
\begin{EQA}[c]
    \P\bigl( \| \Pi \varepsilonv \|_{\norms} \le \rs \bigr)
    \ge 
    1/2,
    \qquad 
    \varepsilonv \sim \ND(0,\Id_{\dimp}) .
\label{epsuvgmrPi}
\end{EQA}
The next result is stated for \( \gms \ge \rs + \uus \), which simplifies the 
formulation.

\begin{theorem}
\label{TxivqLDrgPi}
Let a random vector \( \xiv \) in \( \R^{\dimp} \) fulfill \eqref{expgamgmmn}
and \( \Pi \) follows \( \Pi = \Pi^{\T} \), \( \Pi^{2} \le \Pi \).
Let some \( \uus \) be fixed.
Then for any \( \mus \le 2/3 \) with \( \gms \mus^{-1} - \rs \mus^{-1/2} \ge \uus \), 
\begin{EQA}
    \E \exp\Bigl\{ \frac{\mus}{2} (\| \Pi \xiv \|^{2} - \dimA) \Bigr\}
        \Ind\bigl( \| \Pi^{2} \xiv \|_{\norms} \le \uus \bigr)
    & \le &
    2 \exp( \mus^{2} \vA^{2}/4 ) ,
\label{EexpmusvA}
\end{EQA}    
where \( \vA^{2} = 2 \tr (\Pi^{4}) \).
Moreover, if \( \gms \ge \rs + \uus \),
then for any \( \zz \ge 0 \) 
\begin{EQA}
\label{logPmunuuncPi}
    && \nquad
    \P\bigl( \| \Pi \xiv \|^{2} > \zz, 
        \| \Pi^{2} \xiv \|_{\norms} \le \uus 
    \bigr)
    \\
    & \le &
    \P\bigl( 
        \| \Pi \xiv \|^{2} > \dimA + (2 \vA \xx^{1/2}) \vee (6 \xx), \,
        \| \Pi^{2} \xiv \|_{\norms} \le \uus
    \bigr)
    \le 
    2 \exp( - \xx ).
\label{expxiboPi}
\end{EQA}    

\end{theorem}

\section{A bound for the \( \ell_{2} \)-norm under Bernstein conditions}
\label{SBernstqf}
For comparison, we specify the results to the case considered 
recently in \cite{Ba2010}.
Let \( \zetav \) be a random vector in \( \R^{n} \) whose components \( \zeta_{i} \) 
are independent and satisfy the Bernstein type conditions:
for all \( |\lambda| < c^{-1} \)
\begin{EQA}[c]
\label{xiBernybB}
    \log \E e^{\lambda \zeta_{i}}
    \le 
    \frac{\lambda^{2} \sigma^{2}}{1 - c |\lambda|} . 
\end{EQA}    
Denote \( \xiv = \zetav/(2\sigma) \) and consider 
\( \| \gammav \|_{\norms} = \| \gammav \|_{\infty} \).
Fix \( \gms = \sigma/c \). 
If \( \| \gammav \|_{\norms} \le \gms \), then 
\( 1 - c \gamma_{i}/ (2\sigma) \ge 1/2 \) and
\begin{EQA}[c]
    \log \E \exp\bigl( \gammav^{\T} \xiv \bigr)
    \le 
    \sum_{i} \log \E \exp\Bigl( \frac{\gamma_{i} \zeta_{i}}{2\sigma} \Bigr)
    \le 
    \sum_{i} \frac{|\gamma_{i}/(2\sigma)|^{2} \sigma^{2}}{1 - c \gamma_{i}/(2 \sigma)}
    \le 
    \| \gammav \|^{2}/2 .
\label{logexpggi}
\end{EQA}    
Let also \( S \) be some linear subspace of \( \R^{n} \) with dimension \( \dimA \) 
and \( \Pi_{S} \) denote the projector on \( S \). 
For applying the result of Theorem~\ref{TxivqLDrg}, the value \( \rs \) has to be fixed.
We use that 
the infinity norm \( \| \varepsilonv \|_{\infty} \) concentrates 
around \( \sqrt{2 \log \dimp} \).

\begin{lemma}
\label{LBeboundrs}
It holds for a standard normal vector \( \varepsilonv \in \R^{\dimp} \) 
with \( \rs = \sqrt{2 \log\dimp} \)
\begin{EQA}[c]
    \P\bigl( \| \varepsilonv \|_{\norms} \le \rs \bigr)
    \ge 
    1/2 .
\label{epsvBe}
\end{EQA}    
\end{lemma}
\begin{proof}
By definition
\begin{EQA}[c]
    \P\bigl( \| \varepsilonv \|_{\norms} > \rs \bigr)
    \le 
    \P\bigl( \| \varepsilonv \|_{\infty} > \sqrt{2 \log \dimp} \bigr)
    \le 
    \dimp \P\bigl( |\varepsilon_{1}| > \sqrt{2 \log \dimp} \bigr)
    \le 
    1/2
\label{PsdimpABe}
\end{EQA}    
as required.
\end{proof}

Now the general bound of Theorem~\ref{TxivqLDrg} is applied to bounding the norm of
\( \| \Pi_{S} \xiv \| \). 
For simplicity of formulation we assume that \( \gms \ge \uus + \rs \).

\begin{theorem}
\label{Tchi2Bern}
Let \( S \) be some linear subspace of \( \R^{n} \) with dimension \( \dimA \).
Let \( \gms \ge \uus + \rs \).
If the coordinates \( \zeta_{i} \) of \( \zetav \) are independent and satisfy 
\eqref{xiBernybB}, then for all \( \xx \), 
\begin{EQA}
    \P\bigl( 
        (4 \sigma^{2})^{-1} \| \Pi_{S} \zetav \|^{2} 
        > \dimA + \sqrt{\kappa \xx \dimA} \vee (\kappa \xx), \,
        \| \Pi_{S} \zetav \|_{\infty} \le 2 \sigma \uus
    \bigr)
    & \le &
    2 \exp( - \xx ),
\label{expxiboBe}
\end{EQA}
\end{theorem}

The bound of \cite{Ba2010} reads
\begin{EQA}[c]
\label{Bernmainyb}
    \P\biggl( 
        \| \Pi_{S} \zetav \|_{2} 
        > \bigl( 3 \sigma \, \vee \, \sqrt{6 c u} \bigr) \sqrt{\xx + 3 \dimA} , \,\,
        \| \Pi_{S} \zetav \|_{\infty} \le 2 \sigma \uus
    \biggr)  
    \le 
    e^{- \xx} .
\end{EQA}
As expected, in the region \( \xx \le \xxc \) of Gaussian approximation, the bound 
of Baraud is not sharp and actually quite rough.

\appendix
\section{Proof of Theorem~\ref{TxivG}}
The proof utilizes the following well known fact:
for \( \mu < 1 \)
\begin{EQA}[c]
    \log \E \exp\bigl( \mu \| \xiv \|^{2}/2 \bigr)
    =
    - 0.5 \dimp \log(1-\mu) .
\label{fmupG}
\end{EQA}    
It can be obtained by straightforward calculus. 
Now consider any \( u > 0 \). 
By the exponential Chebyshev inequality
\begin{EQA}
\label{logPmunulg}
    \P\bigl( \| \xiv \|^{2} > \dimp + u \bigr)
    & \le &
    \exp\bigl\{ - \mu (\dimp + u)/2 \bigr\} 
    \E\exp \bigl( \mu \| \xiv \|^{2}/2 \bigr)
    \\
    & = &
    \exp\bigl\{ - \mu (\dimp + u)/2 - (\dimp/2) \log(1 - \mu) \bigr\} .
\end{EQA}
It is easy to see that the value \( \mu = u/(u+\dimp) \) maximizes  
\( \mu (\dimp + u) + \dimp \log(1 - \mu) \) w.r.t. \( \mu \) yielding 
\begin{EQA}[c]
    \mu (\dimp + u) - \dimp \log(1 - \mu)
    =
    u - \dimp \log( 1 + u/\dimp ).
\label{logPmulg}
\end{EQA}    
Further we use that 
\( x - \log(1+x) \ge a_{0} x^{2} \) for \( x \le 1 \)
and \( x - \log(1+x) \ge a_{0} x \) for \( x > 1 \) with 
\( a_{0} = 1 - \log(2) \ge 0.3 \).
This implies with \( x = u/\dimp \) for 
\( u = \sqrt{\kappa \xx \dimp} \) or \( u = \kappa \xx \) and 
\( \kappa = 2/a_{0} < 6.6 \) that
\begin{EQA}
    \P\bigl( \| \xiv \|^{2} \ge \dimp + \sqrt{\kappa \xx \dimp} \vee (\kappa \xx)
    \bigr)
    & \le &
    \exp( - \xx ) 
\label{Ppzv}
\end{EQA}    
as required.

\section{Proof of Theorem~\ref{TexpbLGA}}
The matrix \( \BB^{2} \) can be represented as \( U^{\T} \diag(a_{1},\ldots,a_{\dimp}) U \)
for an orthogonal matrix \( U \). 
The vector \( \tilde{\xiv} = U \xiv \) is also standard normal and 
\( \| \BB \xiv \|^{2} = \tilde{\xiv}^{\T} U \BB^{2} U^{\T} \tilde{\xiv} \).
This means that one can reduce the situation to the case of a diagonal matrix 
\( \BB^{2} = \diag(a_{1},\ldots,a_{\dimp}) \).
We can also assume without loss of generality that 
\( a_{1} \ge a_{2} \ge \ldots \ge a_{\dimp} \).
The expressions for the quantities \( \dimA \) and \( \vA^{2} \) simplifies to
\begin{EQA}
    \dimA
    &=&
    \tr(\BB^{2})
    =
    a_{1} + \ldots + a_{\dimp},
    \\
    \vA^{2}
    &=&
    2 \tr(\BB^{4})
    =
    2 (a_{1}^{2} + \ldots + a_{\dimp}^{2}) .
\label{dimAv2dia}
\end{EQA}    
Moreover, rescaling the matrix \( \BB^{2} \) by \( a_{1} \) reduces the situation to 
the case with \( a_{1} = 1 \).
\begin{lemma}
\label{Lxiv12}
It holds
\begin{EQA}[c]
    \E \| \BB \xiv \|^{2} = \tr(\BB^{2}),
    \qquad 
    \Var\bigl( \| \BB \xiv \|^{2} \bigr) = 2 \tr(\BB^{4}).
\label{ExivVarxiv}
\end{EQA}    
Moreover, for \( \mu < 1 \)
\begin{EQA}[c]
\label{expbLGSA}
    \E \exp\bigl\{ \mu \| \BB \xiv \|^{2}/2 \bigr\}
    =
    \det (1 - \mu \BB^{2})^{-1/2} 
    =
    \prod_{i=1}^{\dimp} (1 - \mu a_{i})^{-1/2}.
\end{EQA}
\end{lemma}
\begin{proof}
If \( \BB^{2} \) is diagonal, then \( \| \BB \xiv \|^{2} = \sum_{i} a_{i} \xi_{i}^{2} \) 
and the summands \( a_{i} \xi_{i}^{2} \) are independent.
It remains to note that \( \E (a_{i} \xi_{i}^{2}) = a_{i} \), 
\( \Var(a_{i} \xi_{i}^{2}) = 2 a_{i}^{2} \), and for \( \mu a_{i} < 1 \),
\begin{EQA}[c]
    \E \exp\bigl\{ \mu a_{i} \xi_{i}^{2}/2 \bigr\} 
    = 
    (1 - \mu a_{i})^{-1/2}
\label{expbLGSAi}
\end{EQA}    
yielding \eqref{expbLGSA}.
\end{proof}

Given \( u \), fix \( \mu < 1 \).
The exponential Markov inequality yields 
\begin{EQA}
    \P\bigl( \| \BB \xiv \|^{2} > \dimA + u \bigr)
    & \le &
    \exp\Bigl\{ - \frac{\mu (\dimA + u)}{2} \Bigr\} 
    \E\exp \Bigl( \frac{\mu \| \BB \xiv \|^{2}}{2} \Bigr)
    \\
    & \le &
    \exp\Bigl\{ 
        -\frac{\mu u}{2} - \frac{1}{2} \sum_{i=1}^{\dimp}  
        \bigl[ \mu a_{i} + \log\bigl( 1 - \mu a_{i} \bigr) \bigr] 
    \Bigr\} .
\label{logPmunugA}
\end{EQA}
We start with the case when \( \xx^{1/2} \le v/3 \).
Then \( u = 2 \xx^{1/2} \vA \) fulfills \( u \le 2v^{2}/3 \).
Define \( \mu = u / \vA^{2} \le 2/3 \) and use that 
\( t + \log(1-t) \ge - t^{2} \) for \( t \le 2/3 \).
This implies
\begin{EQA}
    && \nquad
    \P\bigl( \| \BB \xiv \|^{2} > \dimA + u \bigr)
    \\
    & \le &
    \exp\Bigl\{ 
        -\frac{\mu u}{2} + \frac{1}{2} \sum_{i=1}^{\dimp}  \mu^{2} a_{i}^{2}        
    \Bigr\}
    =
    \exp\bigl(
        -{u^{2}}/{(4 \vA^{2})} 
    \bigr) 
    = 
    e^{-\xx}.
\label{logPmunug1}
\end{EQA}
Next, let \( \xx^{1/2} > v/3 \). 
Set \( \mu = 2/3 \). 
It holds similarly to the above
\begin{EQA}[c]
    \sum_{i=1}^{\dimp} \bigl[ \mu a_{i} + \log\bigl( 1 - \mu a_{i} \bigr) \bigr] 
    \ge 
    - \sum_{i=1}^{\dimp} \mu^{2} a_{i}^{2} 
    \ge 
    - 2 \vA^{2}/9 
    \ge 
    - 2 \xx.
\label{sumI23}
\end{EQA}    
Now, for \( u = 6 \xx \) and \( \mu u / 2 = 2 \xx \), \eqref{logPmunug1} implies
\begin{EQA}[c]
    \P\bigl( \| \BB \xiv \|^{2} > \dimA + u \bigr)
    \le 
    \exp\bigl\{ - \bigl(  2 \xx - \xx  \bigr) \bigr\} 
    = 
    \exp(- \xx)
\label{logPmunug23}
\end{EQA}    
as required.

\section{Proof of Theorem~\ref{TxivqLD}}
The main step of the proof is the following exponential bound.
\begin{lemma}
\label{Lexpxiv} 
Suppose \eqref{expgamgm}.
For any \( \mu < 1 \) with 
\( \gm^{2} > \dimp \mu \), it holds 
\begin{EQA}
\label{Eexp2xi}
    \E \exp\Bigl( \frac{\mu \| \xiv \|^{2}}{2} \Bigr) 
        \Ind\Bigl( \| \xiv \| \le \gm/\mu - \sqrt{\dimp/\mu} \Bigr)
    & \le &
    2 (1 - \mu)^{-\dimp/2} .
\end{EQA}    
\end{lemma}

\begin{proof}
Let \( \varepsilonv \) be a standard normal vector in \( \R^{\dimp} \) and 
\( \uv \in \R^{\dimp} \). 
The bound \( \P\bigl( \| \varepsilonv \|^{2} > \dimp \bigr) \le 1/2 \) implies 
for any vector \( \uv \) and any \( \rr \) with \( \rr \ge \| \uv \| + \dimp^{1/2} \) that 
\( \P\bigl( \| \uv + \varepsilonv \| \le \rr \bigr) \ge 1/2 \).
Let us fix some \( \xiv \) with  \( \| \xiv \| \le \gm/\mu - \sqrt{\dimp/\mu} \) 
and denote by \( \P_{\xiv} \) the conditional probability given \( \xiv \).
It holds with \( c_{p} = (2\pi)^{-\dimp/2} \) 
\begin{EQA}
    && \nquad
    c_{p} \int \exp\Bigl( \gammav^{\T} \xiv - \frac{\| \gammav \|^{2}}{2 \mu}  \Bigr)
        \Ind(\| \gammav \| \le \gm) d\gammav
    \\
    &=&
    c_{p} \exp\bigl( \mu \| \xiv \|^{2} / 2 \bigr)
    \int \exp\Bigl( 
        - \frac{1}{2}  \bigl\| \mu^{-1/2} \gammav - \mu^{1/2} \xiv \bigr\|^{2} 
    \Bigr) \Ind(\mu^{-1/2} \| \gammav \| \le \mu^{-1/2} \gm) d\gammav    
    \\
    & = &
    \mu^{\dimp/2} \exp\bigl( \mu \| \xiv \|^{2} / 2 \bigr) 
    \P_{\xiv}\bigl( \| \varepsilonv + \mu^{1/2} \xiv \| \le \mu^{-1/2} \gm \bigr)
    \\
    & \ge &
    0.5 \mu^{\dimp/2} \exp\bigl( \mu \| \xiv \|^{2} / 2 \bigr) ,
\label{intggvv}
\end{EQA}    
because \( \| \mu^{1/2} \xiv \| + \dimp^{1/2} \le \mu^{-1/2} \gm \).
This implies in view of \( \dimp < \gm^{2}/\mu \) that
\begin{EQA}
    && \nquad
    \exp\bigl( {\mu \| \xiv \|^{2}}/{2} \bigr) 
        \Ind\bigl( \| \xiv \|^{2} \le \gm/\mu - \sqrt{\dimp/\mu} \bigr)
    \\
    & \le &
    2 \mu^{-\dimp/2} c_{p}
    \int \exp\Bigl( \gammav^{\T} \xiv - \frac{\| \gammav \|^{2}}{2\mu} \Bigr)
        \Ind(\| \gammav \| \le \gm) d\gammav .
\label{expxiv1cp}
\end{EQA}    
Further, by \eqref{expgamgm}
\begin{EQA}
    && 
    \nquad
    c_{p} \E \int \exp\Bigl( \gammav^{\T} \xiv - \frac{1}{2\mu} \| \gammav \|^{2} \Bigr)
        \Ind(\| \gammav \| \le \gm) d\gammav
    \\
    & \le &
    c_{p} \int \exp\Bigl( - \frac{\mu^{-1} - 1}{2} \| \gammav \|^{2} \Bigr)
        \Ind(\| \gammav \| \le \gm) d\gammav
    \\
    & \le &
    c_{p} \int \exp\Bigl( - \frac{\mu^{-1} - 1}{2} \| \gammav \|^{2} \Bigr) d \gammav
    \\
    & \le &
    (\mu^{-1} - 1)^{- \dimp/2}
\label{nununu}
\end{EQA}    
and \eqref{Eexp2xi} follows.
\end{proof}

Due to this result, the scaled squared norm \( \mu \| \xiv \|^{2}/2 \) after a proper
truncation possesses the same exponential moments as in the Gaussian case.
A straightforward implication is the probability bound 
\( \P\bigl( \| \xiv \|^{2} > \dimp + u \bigr) \) for moderate values \( u \).
Namely, given \( u > 0 \), define \( \mu = u/(u+\dimp) \).
This value optimizes the inequality \eqref{logPmunulg} in the Gaussian case.
Now we can apply a similar bound under the constraints 
\( \| \xiv \| \le \gm/\mu - \sqrt{\dimp/\mu} \).
Therefore, the bound is only meaningful if 
\( \sqrt{u + \dimp} \le \gm/\mu - \sqrt{\dimp/\mu} \) with \( \mu = u/(u+\dimp) \),
or, with \( w = \sqrt{u/\dimp} \le \wwc \); see \eqref{wc212}.
    
The largest value \( u \) for which this constraint is still 
valid, is given by \( \dimp + u = \yyc^{2} \).
Hence, \eqref{Eexp2xi} yields for \( \dimp + u \le \yyc^{2} \)
\begin{EQA}
    && \nquad
    \P\bigl( \| \xiv \|^{2} > \dimp + u, \| \xiv \| \le \yyc \bigr)
    \\
    & \le &
    \exp\Bigl\{ - \frac{\mu (\dimp + u)}{2} \Bigr\} 
    \E\exp \Bigl( \frac{\mu \| \xiv \|^{2}}{2} \Bigr)
    \Ind\Bigl( \| \xiv \| \le \gm/\mu - \sqrt{\dimp/\mu} \Bigr)
    \\
    & \le &
    2 \exp\bigl\{ - 0.5 \bigl[ \mu (\dimp + u) + \dimp \log( 1 - \mu ) \bigr] 
    \bigr\} 
    \\
    &=&
    2 \exp\bigl\{ 
        - 0.5 \bigl[ u - \dimp \log( 1 + u/\dimp ) \bigr] 
    \bigr\} .
\label{logPmunu}
\end{EQA}
Similarly to the Gaussian case, this implies with
\( \kappa = 6.6 \) that
\begin{EQA}
    \P\bigl( \| \xiv \| \ge \dimp + \sqrt{\kappa \xx \dimp} \vee (\kappa \xx), 
        \| \xiv \| \le \yyc
    \bigr)
    & \le &
    2 \exp( - \xx ) .
\label{Ppzv}
\end{EQA}    
The Gaussian case means that \eqref{expgamgm} holds with \( \gm = \infty \)
yielding \( \yyc = \infty \). 
In the non-Gaussian case with a finite \( \gm \), we have to accompany the moderate 
deviation bound with a large deviation bound 
\( \P\bigl( \| \xiv \| > \yy \bigr) \) for \( \yy \ge \yyc \).
This is done by combining the bound \eqref{Eexp2xi} with
the standard slicing arguments.

\begin{lemma}
\label{Lexpxi2}
Let \( \mud \le \gm^{2}/\dimp \).
Define \( \yyd = \gm/\mud - \sqrt{\dimp/\mud} \) and 
\( \gmd = \mud \yyd = \gm - \sqrt{\mud \dimp} \).
It holds for \( \yy \ge \yyd \)
\begin{EQA}
\label{Pexp2xi}
    \P\bigl( \| \xiv \| > \yy \bigr)
    & \le &
    8.4 (1 - \gmd/\yy)^{-\dimp/2} \exp\bigl( - \gmd \yy/2 \bigr)
    \\
    & \le &
    8.4 \exp\bigl\{ - \xxd -  \gmd (\yy - \yyd)/2 \bigr\}.
\label{Pexp2xiy}
\end{EQA}    
with \( \xxd \) defined by 
\begin{EQA}[c]
    2 \xxd 
    = 
    \mud \yyd^{2} + \dimp \log(1 - \mud) 
    =
    \gm^{2}/\mud - \dimp + \dimp \log(1 - \mud).
\label{xxdyydmud}
\end{EQA}    
\end{lemma}

\begin{proof}
Consider the growing sequence \( \yy_{k} \) with  \( \yy_{1} = \yy \) 
and \( \gmd \yy_{k+1} = \gmd \yy + k \).
Define also \( \mu_{k} = \gmd/ \yy_{k} \).
In particular, \( \mu_{k} \le \mu_{1} = \gmd / \yy \).
Obviously
\begin{EQA}
    \P\bigl( \| \xiv \| > \yy \bigr)
    &=&
    \sum_{k=1}^{\infty} 
        \P\bigl( \| \xiv \| > \yy_{k}, \| \xiv \| \le \yy_{k+1}
        \bigr).
\label{zzkkp1}
\end{EQA}    
Now we try to evaluate every slicing probability in this expression.
We use that 
\begin{EQA}
    \mu_{k+1} \yy_{k}^{2}
    &=&
    \frac{(\gmd \yy  + k - 1)^{2}}{\gmd \yy  + k} 
    \ge 
    \gmd \yy + k - 2 ,
\label{nukk1}
\end{EQA}    
and also \( \gm/\mu_{k} - \sqrt{\dimp / \mu_{k}} \ge \yy_{k} \) because
\( \gm - \gmd = \sqrt{\mud \dimp} > \sqrt{\mu_{k} \dimp} \) and 
\begin{EQA}[c]
    \gm/\mu_{k} - \sqrt{\dimp / \mu_{k}} - \yy_{k}
    =
    \mu_{k}^{-1} (\gm - \sqrt{\mu_{k} \dimp} - \gmd)
    \ge 0.
\label{gmgmmyy}
\end{EQA}    
Hence by \eqref{Eexp2xi} 
\begin{EQA}
    \P\Bigl( \| \xiv \| > \yy \Bigr)
    & \le &
    \sum_{k=1}^{\infty} 
     \P\Bigl( 
        \| \xiv \| > \yy_{k}, 
        \| \xiv \| \le \yy_{k+1}
     \Bigr)
    \\
    & \le &
    \sum_{k=1}^{\infty} \exp\Bigl( - \frac{\mu_{k+1} \yy_{k}^{2}}{2} \Bigr)
    \E \exp\Bigl( \frac{\mu_{k+1} \| \xiv \|^{2}}{2} \Bigr)
        \Ind\bigl( \| \xiv \| \le \yy_{k+1} \bigr)
    \\
    & \le &
    \sum_{k=1}^{\infty} 2 \bigl( 1 - \mu_{k+1} \bigr)^{-\dimp/2}
    \exp\Bigl( - \frac{\mu_{k+1} \yy_{k}^{2}}{2} \Bigr)
    \\
    & \le &
    2 \bigl( 1 - \mu_{1} \bigr)^{- \dimp/2}
    \sum_{k=1}^{\infty} \exp\Bigl( - \frac{\gmd \yy + k - 2}{2} \Bigr) 
    \\
    & = &
    2 e^{1/2} (1 - e^{-1/2})^{-1} 
    (1 - \mu_{1})^{-\dimp/2} \exp\bigl( - \gmd \yy/2 \bigr) 
    \\
    & \le &
    8.4 (1 - \mu_{1})^{-\dimp/2} \exp\bigl( - \gmd \yy/2 \bigr)
\label{Psumnuk}
\end{EQA}    
and the first assertion follows.
For  \( \yy = \yyd \), it holds 
\begin{EQA}[c]
    \gmd \yyd + \dimp \log(1 - \mud)
    =
    \mud \yyd^{2} + \dimp \log(1 - \mud)
    =
    2 \xxd
\label{gmcyyc2xxc}
\end{EQA}    
and \eqref{Pexp2xi} implies 
\( \P\bigl( \| \xiv \| > \yyd \bigr) \le 8.4 \exp(- \xxd) \).
Now observe that the function 
\( f(\yy) = \gmd \yy/2 + (\dimp/2) \log \bigl( 1 - \gmd/\yy \bigr) \)
fulfills \( f(\yyd) = \xxd \) and \( f'(\yy) \ge \gmd/2 \) yielding 
\( f(\yy) \ge \xxd + \gmd (\yy - \yyd)/2 \).
This implies \eqref{Pexp2xiy}.
\end{proof}

The statements of the theorem are obtained by applying the lemmas with 
\( \mud = \muc = \wwc^{2}/(1 + \wwc^{2}) \).
This also implies 
\( \yyd = \yyc \), \( \xxd = \xxc \), and 
\( \gmd = \gmc = \gm - \sqrt{\muc \dimp} \);
cf. \eqref{zcgmp}. 

\section{Proof of Theorem~\ref{TxivqLDA}}
The main steps of the proof are similar to the proof of Theorem~\ref{TxivqLD}.
\begin{lemma}
\label{Lexpxig} 
Suppose \eqref{expgamgm}.
For any \( \mu < 1 \) with
\( \gm^{2}/\mu \ge \dimA \), it holds 
\begin{EQA}
\label{Eexp2xig}
    \E \exp\bigl( {\mu \| \BB \xiv \|^{2}}/{2} \bigr)
        \Ind\bigl( \| \BB^{2} \xiv \| \le \gm/\mu - \sqrt{\dimA/\mu} \bigr)
    & \le &
    2 {\det(\Id_{\dimp} - \mu \BB^{2})^{-1/2}} .
\end{EQA}    
\end{lemma}

\begin{proof}
With \( c_{\dimp}(\BB) = \bigl( 2 \pi \bigr)^{-\dimp/2} \det(\BB^{-1}) \)
\begin{EQA}
    && \nquad
    c_{p}(\BB) \int \exp\Bigl( \gammav^{\T} \xiv - \frac{1}{2 \mu} \| \BB^{-1} \gammav \|^{2} \Bigr)
        \Ind(\| \gammav \| \le \gm) d\gammav
    \\
    &=&
    c_{p}(\BB) \exp\Bigl( \frac{\mu \| \BB \xiv \|^{2}}{2} \Bigr)
    \int \exp\Bigl( 
        - \frac{1}{2}  \bigl\| \mu^{1/2} \BB \xiv - \mu^{-1/2} \BB^{-1} \gammav \bigr\|^{2} 
    \Bigr) \Ind(\| \gammav \| \le \gm) d\gammav    
    \\
    &=&
    \mu^{\dimp/2} \exp\Bigl( \frac{\mu \| \BB \xiv \|^{2}}{2} \Bigr) 
    \P_{\xiv}\bigl( 
        \| \mu^{-1/2} \BB \varepsilonv + \BB^{2} \xiv \| \le \gm / \mu
    \bigr),
\label{intggvvg}
\end{EQA}
where \( \varepsilonv \) denotes a standard normal vector in \( \R^{\dimp} \)
and \( \P_{\xiv} \) means the conditional probability given \( \xiv \).
Moreover, for any \( \uv \in \R^{\dimp} \) and \( \rr \ge \dimA^{1/2} + \| \uv \| \), 
it holds in view of \( \P \bigl( \| \BB \varepsilonv \|^{2} > \dimA \bigr) \le 1/2 \)
\begin{EQA}
    \P\bigl( \| \BB \varepsilonv - \uv \| \le \rr \bigr)
    & \ge &
    \P\bigl( \| \BB \varepsilonv \| \le \sqrt{\dimA} \bigr)
    \ge 
    1/2 .
\label{Arepsv}
\end{EQA}    
This implies
\begin{EQA}
    && 
    \nquad
    \exp\Bigl( \mu \| \BB \xiv \|^{2} / 2 \Bigr) 
        \Ind\bigl( \| \BB^{2} \xiv \| \le \gm / \mu - \sqrt{\dimA/\mu} \bigr)
    \\
    & \le &
    2 \mu^{- \dimp/2} c_{p}(\BB)
    \int \exp\Bigl( \gammav^{\T} \xiv - \frac{1}{2 \mu} \| \BB^{-1} \gammav \|^{2} \Bigr)
        \Ind(\| \gammav \| \le \gm) d\gammav .
\label{expxiv1cpg}
\end{EQA}    
Further, by \eqref{expgamgm}
\begin{EQA}
    && 
    \nquad
    c_{p}(\BB) \E \int \exp\Bigl( \gammav^{\T} \xiv - \frac{1}{2\mu} \| \BB^{-1} \gammav \|^{2} \Bigr)
        \Ind(\| \gammav \| \le \gm) d\gammav
    \\
    & \le &
    c_{p}(\BB) \int \exp\Bigl( 
        \frac{\| \gammav \|^{2}}{2} - \frac{1}{2\mu} \| \BB^{-1} \gammav \|^{2} 
    \Bigr) d\gammav
    \\
    & \le &
    \det(\BB^{-1}) \det(\mu^{-1} \BB^{-2} - \Id_{\dimp})^{-1/2}
    =
    \mu^{p/2} \det( \Id_{\dimp} - \mu \BB^{2} )^{-1/2}
\label{nununug}
\end{EQA}    
and \eqref{Eexp2xig} follows.
\end{proof}

Now we evaluate the probability \( \P\bigl( \| \BB \xiv \| > \yy \bigr) \) for moderate 
values of \( \yy \).
\begin{lemma}
\label{Lexplotr}
Let \( \mud < 1 \wedge (\gm^{2}/\dimA) \). 
With \( \yyd = \gm/\mud - \sqrt{\dimA/\mud} \), it holds 
for any \( u > 0 \)
\begin{EQA}
    && \nquad
    \P\bigl( \| \BB \xiv \|^{2} > \dimA + u, \| \BB^{2} \xiv \| \le \yyd \bigr)
    \\
    & \le &
    2 \exp\bigl\{ 
        - 0.5\mud (\dimA + u) - 0.5 \log \det (\Id_{\dimp} - \mud \BB^{2}) 
    \bigr\} .
\label{exptrxiq}
\end{EQA}    
In particular, if \( \BB^{2} \) is diagonal, that is, 
\( \BB^{2} = \diag\bigl( a_{1},\ldots,a_{\dimp} \bigr) \), then 
\begin{EQA}
    && \nquad
    \P\bigl( \| \BB \xiv \|^{2} > \dimA + u, \| \BB^{2} \xiv \| \le \yyd \bigr)
    \\
    & \le &
    2 \exp\Bigl\{ 
        -\frac{\mud u}{2} - \frac{1}{2} \sum_{i=1}^{\dimp}  
        \bigl[ \mud a_{i} + \log\bigl( 1 - \mud a_{i} \bigr) \bigr] 
    \Bigr\} .
\label{exptrxiqq}
\end{EQA}
\end{lemma}

\begin{proof}
The exponential Chebyshev inequality and \eqref{Eexp2xig} imply
\begin{EQA}
    && \nquad
    \P\bigl( \| \BB \xiv \|^{2} > \dimA + u, \| \BB^{2} \xiv \| \le \yyd \bigr)
    \\
    & \le &
    \exp\Bigl\{ - \frac{\mud (\dimA + u)}{2} \Bigr\} 
    \E\exp \Bigl( \frac{\mud \| \BB \xiv \|^{2}}{2} \Bigr)
    \Ind\bigl( \| \BB^{2} \xiv \| \le \gm/\mud - \sqrt{\dimA/\mud} \bigr)
    \\
    & \le &
    2 \exp \bigl\{ - 0.5\mud (\dimA + u) - 0.5 \log \det (\Id_{\dimp} - \mud \BB^{2}) \bigr\}.
\label{logPmunug}
\end{EQA}
Moreover, the standard change-of-basis arguments allow us to reduce the problem to the case of a 
diagonal matrix \( \BB^{2} = \diag\bigl( a_{1},\ldots,a_{\dimp} \bigr) \) where 
\( 1 = a_{1} \ge a_{2} \ge \ldots\ge a_{\dimp} > 0 \).
Note that \( \dimA = a_{1} + \ldots + a_{\dimp} \).
Then the claim \eqref{exptrxiq} can be written in the form \eqref{exptrxiqq}.
\end{proof}

Now we evaluate a large deviation probability that 
\( \| \BB \xiv \| > \yy \) for a large \( \yy \).
Note that the condition \( \| \BB^{2} \|_{\infty} \le 1 \) implies 
\( \| \BB^{2} \xiv \| \le \| \BB \xiv \| \).
So, the bound \eqref{exptrxiq} continues to hold when \( \| \BB^{2} \xiv \| \le \yyd \) is 
replaced by \( \| \BB \xiv \| \le \yyd \).
\begin{lemma}
\label{Lexpxi2g}
Let \( \mud < 1 \) and \( \mud \dimA < \gm^{2} \). 
Define \( \gmd \) by \( \gmd = \gm - \sqrt{\mud \dimA} \).
For any \( \yy \ge \yyd \eqdef \gmd/\mud \), it holds 
\begin{EQA}
    \P\bigl( \| \BB \xiv \| > \yy \bigr)
    & \le &
    8.4 \det\{ \Id_{\dimp} - (\gmd/\yy) \BB^{2} \}^{-1/2}
    \exp\bigl( - \gmd \yy/2 \bigr).
    \\
    & \le &
    8.4 \exp\bigl( - \xxd - \gmd (\yy - \yyd)/2 \bigr) ,
\label{Pexp2xig}
\end{EQA}    
where \( \xxd \) is defined by 
\begin{EQA}
    2 \xxd
    & = &
    \gmd \yyd + \log \det\{ \Id_{\dimp} - (\gmd/\yyd) \BB^{2} \} .
\label{xxdyydA}
\end{EQA}
\end{lemma}

\begin{proof}
The slicing arguments of Lemma~\ref{Lexpxi2} apply here in the same manner.
One has to replace \( \| \xiv \| \) by \( \| \BB \xiv \| \) and 
\( (1 - \mu_{1})^{-\dimp/2} \) by \( \det \{ \Id_{\dimp} - (\gmd/\yy) \BB^{2} \}^{-1/2} \).
We omit the details.
In particular, with \( \yy = \yyd = \gmd/\mu \), this yields
\begin{EQA}
    \P\bigl( \| \BB \xiv \| > \yyd \bigr)
    & \le &
    8.4 \exp (- \xxd).
\label{Pexp2xigmmc}
\end{EQA}
Moreover, for the function 
\( f(\yy) = \gmd \yy + \log \det \{ \Id_{\dimp} - (\gmd/\yy) \BB^{2} \} \), it holds
\( f'(\yy) \ge \gmd \) and hence,
\( f(\yy) \ge f(\yyd) + \gmd (\yy - \yyd) \) for \( \yy > \yyd \).
This implies \eqref{Pexp2xig}.
\end{proof}

One important feature of the results of Lemma~\ref{Lexplotr} and Lemma~\ref{Lexpxi2g} is 
that the value \( \mud < 1 \wedge (\gm^{2}/\dimA) \) can be selected arbitrarily.
In particular, for \( \yy \ge \yyc \), Lemma~\ref{Lexpxi2g} with \( \mud = \muc \) 
yields the large deviation probability \( \P\bigl( \| \BB \xiv \| > \yy \bigr) \).
For bounding the probability 
\( \P\bigl( \| \BB \xiv \|^{2} > \dimA + u, \, \| \BB \xiv \| \le \yyc \bigr) \),
we use the inequality \( \log(1-t) \ge - t - t^{2} \)
for \( t \le 2/3 \).
It implies  for \( \mu \le 2/3 \) that
\begin{EQA}
    && \nquad
    - \log \P\bigl( \| \BB \xiv \|^{2} > \dimA + u, \| \BB \xiv \| \le \yyc \bigr)
    \\
    & \ge &
    \mu (\dimA + u)
    + \sum_{i=1}^{\dimp} \log\bigl( 1 - \mu a_{i} \bigr) 
    \\
    & \ge &
    \mu (\dimA + u) - \sum_{i=1}^{\dimp} (\mu a_{i} + \mu^{2} a_{i}^{2}) 
    \ge  
    \mu u - \mu^{2} \vA^{2}/2 .
\label{sumI23A}
\end{EQA}
Now we distinguish between \( \muc = 2/3 \) and \( \muc < 2/3 \) starting with 
\( \muc = 2/3 \).
The bound \eqref{sumI23A} with
\( \mu = 2/3 \) and 
with \( u = (2 \vA \xx^{1/2}) \vee (6 \xx) \) yields 
\begin{EQA}
    \P\bigl( \| \BB \xiv \|^{2} > \dimA + u, \, \| \BB \xiv \| \le \yyc \bigr)
    & \le &
    2 \exp (- \xx) ;
\label{AxivudA}
\end{EQA}
see the proof of Theorem~\ref{TexpbLGA} for the Gaussian case.

Now consider \( \muc < 2/3 \).
For \( \xx^{1/2} \le \muc \vA / 2 \), use \( u = 2 \vA \xx^{1/2} \) and 
\( \mud = u / \vA^{2} \). 
It holds \( \mud = u / \vA^{2} \le \muc \) and \( u^{2}/(4 \vA^{2}) = \xx \) yielding the 
desired bound by \eqref{sumI23A}. 
For \( \xx^{1/2} > \muc \vA / 2 \), we select again \( \mud = \muc \).
It holds with \( u = 4 \muc^{-1} \xx \) that 
\( \muc u / 2 - \muc^{2} \vA^{2}/4 \ge 2\xx - \xx = \xx \). 
This completes the proof. 

\section{Proof of Theorem~\ref{TxivqLDrg}}
The arguments behind the result are the same as in the one-norm case of 
Theorem~\ref{TxivqLD}.
We only outline the main steps.

\begin{lemma}
\label{Lexpxigm} 
Suppose \eqref{expgamgmmn} and \eqref{epsuvgmr}.
For any \( \mu < 1 \) with \( \gms > \mu^{1/2} \rs \), it holds 
\begin{EQA}
\label{Eexp2xigm}
    \E \exp\bigl( \mu \| \xiv \|^{2} / 2 \bigr)
        \Ind\bigl( \| \xiv \|_{\norms} \le \gms/\mu - \rs/\mu^{1/2} \bigr)
    & \le &
    2 (1 - \mu)^{-\dimp/2} .
\end{EQA}    
\end{lemma}

\begin{proof}
Let \( \varepsilonv \) be a standard normal vector in \( \R^{\dimp} \) and 
\( \uv \in \R^{\dimp} \). 
Let us fix some \( \xiv \) with  \( \mu^{1/2} \| \xiv \|_{\norms} \le \mu^{-1/2} \gms - \rs \) 
and denote by \( \P_{\xiv} \) the conditional probability given \( \xiv \).
It holds by \eqref{epsuvgmr} with \( c_{p} = (2\pi)^{-\dimp/2} \) 
\begin{EQA}
    && \nquad
    c_{p} \int \exp\Bigl( \gammav^{\T} \xiv - \frac{1}{2 \mu} \| \gammav \|^{2} \Bigr)
        \Ind(\| \gammav \|_{\norms} \le \gms) d\gammav
    \\
    &=&
    c_{p} \exp\bigl( \mu \| \xiv \|^{2} / 2 \bigr)
    \int \exp\Bigl( 
        - \frac{1}{2}  \bigl\| \mu^{1/2} \xiv - \mu^{-1/2} \gammav \bigr\|^{2} 
    \Bigr) \Ind( \| \mu^{-1/2} \gammav \|_{\norms} \le \mu^{-1/2} \gms) d\gammav    
    \\
    & = &
    \mu^{\dimp/2} \exp\bigl( \mu \| \xiv \|^{2} / 2 \bigr) 
    \P_{\xiv}\bigl( \| \varepsilonv - \mu^{1/2} \xiv \|_{\norms} \le \mu^{-1/2} \gms \bigr)
    \\
    & \ge &
    0.5 \mu^{\dimp/2} \exp\bigl( \mu \| \xiv \|^{2} / 2 \bigr) .
\label{intggvvgm}
\end{EQA}    
This implies 
\begin{EQA}
    && \nquad
    \exp\Bigl( \frac{\mu \| \xiv \|^{2}}{2} \Bigr) 
        \Ind\bigl( \| \xiv \|_{\norms} \le \gms/\mu - \rs/\mu^{1/2} \bigr)
    \\
    & \le &
    2 \mu^{-\dimp/2} c_{p}
    \int \exp\Bigl( \gammav^{\T} \xiv - \frac{1}{2\mu} \| \gammav \|^{2} \Bigr)
        \Ind(\| \gammav \|_{\norms} \le \gms) d\gammav .
\label{expxiv1cpgm}
\end{EQA}    
Further, by \eqref{expgamgmmn}
\begin{EQA}
    && 
    \nquad
    c_{p} \E \int \exp\Bigl( \gammav^{\T} \xiv - \frac{1}{2\mu} \| \gammav \|^{2} \Bigr)
        \Ind(\| \gammav \|_{\norms} \le \gms) d\gammav
    \\
    & \le &
    c_{p} \int \exp\Bigl( - \frac{\mu^{-1} - 1}{2} \| \gammav \|^{2} \Bigr) d \gammav
    \le 
    (\mu^{-1} - 1)^{- \dimp/2}
\label{nununugm}
\end{EQA}    
and \eqref{Eexp2xigm} follows.
\end{proof}

As in the Gaussian case, \eqref{Eexp2xigm} implies for \( \zz > \dimp \) with
\( \mu = \mu(\zz) = (\zz - \dimp)/\zz \) the bounds \eqref{logPmunuunc} and 
\eqref{expxibon}. 
Note that the value \( \mu(\zz) \) clearly grows with \( \zz \) from zero to one, while  
\( \gms/\mu(\zz) - \rs/\mu^{1/2}(\zz) \) is strictly decreasing.
The value \( \zzs \) is defined exactly as the point where 
\( \gms/\mu(\zz) - \rs/\mu^{1/2}(\zz) \) crosses \( \uus \), so that 
\( \gms/\mu(\zz) - \rs/\mu^{1/2}(\zz) \ge \uus \) for \( \zz \le \zzs \).

For \( \zz > \zzs \), the choice \( \mu = \mu(\yy) \) conflicts with 
\( \gms/\mu(\zz) - \rs/\mu^{1/2}(\zz) \ge \uus \). 
So, we apply \( \mu = \mus \) yielding by the Markov inequality
\begin{EQA}[c]
    \P\bigl( \| \xiv \|^{2} > \zz, \, \| \xiv \|_{\norms} \le \uus \bigr)
    \le 
    2 \exp\bigl\{ - \mus \zz/2 - (\dimp/2) \log(1 - \mus) \bigr\} ,
\label{Pxivsyyps}
\end{EQA}    
and the assertion follows.

\section{Proof of Theorem~\ref{TxivqLDrgPi}}
Arguments from the proof of Lemmas~\ref{Lexpxig} and \ref{Lexpxigm} yield in view of 
\( \gms \mus^{-1} - \rs \mus^{-1/2} \ge \uus \)
\begin{EQA}
\label{Eexp2xigPi}
    && \nquad
    \E \exp\bigl\{ \mus \| \Pi \xiv \|^{2}/2 \Bigr\}
        \Ind\bigl( \| \Pi^{2} \xiv \|_{\norms} \le \uus \bigr)
    \\
    & \le &        
    \E \exp\bigl( {\mus \| \Pi \xiv \|^{2}}/{2} \bigr)
        \Ind\bigl( \| \Pi^{2} \xiv \|_{\norms} \le \gms / \mus - \dimA/\mus^{1/2} \bigr)
    \\
    & \le &
    2 {\det(\Id_{\dimp} - \mus \Pi^{2})^{-1/2}} .
\end{EQA}
Now the inequality \( \log(1-t) \ge - t - t^{2} \)
for \( t \le 2/3 \) implies
\begin{EQA}[c]
    - \log \det(\Id_{\dimp} - \mus \Pi^{2}) 
    \le 
    \mus \dimA + \mus^{2} \vA^{2}/2
\label{EexpmusvAte}
\end{EQA}    
cf. \eqref{sumI23A}; the assertion \eqref{EexpmusvA} follows.

\bibliography{exp_ts,listpubm-with-url}

\begin{thebibliography}{}

\bibitem[Baraud, 2010]{Ba2010}
Baraud, Y. (2010).
\newblock {A Bernstein-type inequality for suprema of random processes with
  applications to model selection in non-Gaussian regression.}
\newblock {\em Bernoulli}, 16(4):1064--1085.

\bibitem[Boucheron and Massart, 2011]{BoMa2011}
Boucheron, S. and Massart, P. (2011).
\newblock A high-dimensional wilks phenomenon.
\newblock {\em Probability Theory and Related Fields}, 150:405--433.
\newblock 10.1007/s00440-010-0278-7.

\bibitem[Bretagnolle, 1999]{Bret1999}
Bretagnolle, J. (1999).
\newblock {A new large deviation inequality for $U$-statistics of order $2$.}
\newblock {\em ESAIM, Probab. Stat.}, 3:151--162.

\bibitem[Fan and Huang, 2005]{FaHu2005}
Fan, J. and Huang, T. (2005).
\newblock {Profile likelihood inferences on semiparametric varying-coefficient
  partially linear models.}
\newblock {\em Bernoulli}, 11(6):1031--1057.

\bibitem[Fan et~al., 2001]{FaZh2001}
Fan, J., Zhang, C., and Zhang, J. (2001).
\newblock {Generalized likelihood ratio statistics and Wilks phenomenon.}
\newblock {\em Ann. Stat.}, 29(1):153--193.

\bibitem[Gin\'e et~al., 2000]{Gine2000}
Gin\'e, E., Lata{\l}a, R., and Zinn, J. (2000).
\newblock {Exponential and moment inequalities for $U$-statistics.}
\newblock {Gin\'e, Evarist (ed.) et al., High dimensional probability II. 2nd
  international conference, Univ. of Washington, DC, USA, August 1-6, 1999.
  Boston, MA: Birkh\"auser. Prog. Probab. 47, 13-38 (2000).}

\bibitem[G\"otze and Tikhomirov, 1999]{GoTi1999}
G\"otze, F. and Tikhomirov, A. (1999).
\newblock {Asymptotic distribution of quadratic forms.}
\newblock {\em Ann. Probab.}, 27(2):1072--1098.

\bibitem[Horv\'ath and Shao, 1999]{HoSh1999}
Horv\'ath, L. and Shao, Q.-M. (1999).
\newblock {Limit theorems for quadratic forms with applications to Whittle's
  estimate.}
\newblock {\em Ann. Appl. Probab.}, 9(1):146--187.

\bibitem[Houdr\'e and Reynaud-Bouret, 2003]{HouRe2003}
Houdr\'e, C. and Reynaud-Bouret, P. (2003).
\newblock {Exponential inequalities, with constants, for U-statistics of order
  two.}
\newblock {Gin\'e, Evariste (ed.) et al., Stochastic inequalities and
  applications. Selected papers presented at the Euroconference on ``Stochastic
  inequalities and their applications", Barcelona, June 18--22, 2002. Basel:
  Birkh\"auser. Prog. Probab. 56, 55-69 (2003).}

\bibitem[Massart, 2007]{massart2003}
Massart, P. (2007).
\newblock {\em {Concentration inequalities and model selection. Ecole d'Et\'e
  de Probabilit\'es de Saint-Flour XXXIII-2003}}.
\newblock {Lecture Notes in Mathematics. Springer}.

\end{thebibliography}
\end{document}